\documentclass[12pt]{article}

\usepackage{amsmath,amsthm,amsfonts,amssymb,amscd,epsf,epsfig,psfrag,enumerate}
\usepackage[mathscr]{eucal}
\usepackage{graphicx,epstopdf}
\usepackage{float}

\title{Measured lamination spaces for surface pairs}

\author{Ulrich Oertel}

\date{Revised September 2013}

\newtheorem{thm}{Theorem}[section] \newtheorem{lemma}[thm]{Lemma}

\newtheorem{corollary}[thm]{Corollary}

\newtheorem{proposition}[thm]{Proposition}
 \newtheorem*{claim*}{Claim}

 \theoremstyle{definition}
\newtheorem{defn}[thm]{Definition}
\newtheorem{defns}[thm]{Definitions} 
 \newtheorem{ex}[thm]{Example}

\theoremstyle{remark}

\begin{document}

\maketitle

%\tableofcontents

\def\HDS{half-disk sum}

\def\Length{\text{Length}}

\def\Area{\text{Area}}
\def\Im{\text{Im}}
\def\im{\text{Im}}
\def\cl{\text{cl}}
\def\rel{\text{ rel }}
\def\irred{irreducible}
\def\half{spinal pair }
\def\spinal{\half}
\def\spinals{\halfs}
\def\halfs{spinal pairs }
\def\reals{\mathbb R}
\def\rationals{\mathbb Q}
\def\complex{\mathbb C}
\def\naturals{\mathbb N}
\def\integers{\mathbb Z}
\def\id{\text{id}}
\def\Chi{\raise1.5pt \hbox{$\chi$}}

\def\proj{P}
\def\hyp {\hbox {\rm {H \kern -2.8ex I}\kern 1.15ex}}

\def\Diff{\text{Diff}}

\def\weight#1#2#3{{#1}\raise2.5pt\hbox{$\centerdot$}\left({#2},{#3}\right)}
\def\intr{{\rm int}}
\def\inter{\ \raise4pt\hbox{$^\circ$}\kern -1.6ex}
\def\Cal{\cal}
\def\from{:}
\def\inverse{^{-1}}
\def\Max{{\rm Max}}
\def\Min{{\rm Min}}
\def\fr{{\rm fr}}
\def\embed{\hookrightarrow}
\def\Genus{{\rm Genus}}
\def\Z{Z}
\def\X{X}

\def\roster{\begin{enumerate}}
\def\endroster{\end{enumerate}}
\def\intersect{\cap}
\def\definition{\begin{defn}}
\def\enddefinition{\end{defn}}
\def\subhead{\subsection\{}
\def\theorem{thm}
\def\endsubhead{\}}
\def\head{\section\{}
\def\endhead{\}}
\def\example{\begin{ex}}
\def\endexample{\end{ex}}
\def\ves{\vs}
\def\mZ{{\mathbb Z}}
\def\M{M(\Phi)}
\def\bdry{\partial}
\def\hop{\vskip 0.15in}
\def\hip{\vskip0.1in}
\def\mathring{\inter}
\def\trip{\vskip 0.09in}
\def\PML{\mathscr{PML}}
\def\H{\mathscr{H}}
\def\C{\mathscr{C}}
\def\S{\mathscr{S}}
\def\T{\mathscr{T}}
\def\E{\mathscr{E}}
\def\K{\mathscr{K}}
\def\BL{\mathscr{BL}}
\def\L{\mathscr{L}}
\def\suchthat{|}
\newcommand\invlimit{\varprojlim}
\newcommand\congruent{\equiv}
\newcommand\modulo[1]{\pmod{#1}}
\def\ML{\mathscr{ML}}
\def\Stack{\mathscr{T}}
\def\M{\mathscr{M}}
\def\A{\mathscr{A}}
\def\union{\cup}
\def\atlas{\mathscr{A}}
\def\interior{\text{Int}}
\def\frontier{\text{Fr}}
\def\composed{\circ}
\def\PC{\mathscr{PC}}
\def\PM{\mathscr{PM}}
\def\D{\mathscr{D}}
\def\WC{\mathscr{WC}}
\def\W{\mathscr{W}}
\def\V{\mathscr{V}}
\def\PV{\mathscr{PV}}

\def\FDM{\mathscr{FDM}}
\def\I{\mathscr{I}}
\def\split{\prec}
\def\pinch{\succ}

\centerline{\bf PRELIMINARY}
\begin{abstract} We calculate a projective space of
essential measured laminations in a surface pair, which will be used in another paper to help describe
spaces of
``finite height laminations."
\end{abstract}

\section{Introduction.}

This paper is intended as a preparation for a study of ``finite depth" or ``finite height" essential measured laminations in surfaces, \cite{UO:Depth}.  In particular, we wish
to describe a suitably projectivized space of finite depth measured laminations.  

\begin{defn}  For simplicity, assume $F$ is an orientable closed surface with $\chi(F)<0$.  A {\it finite depth essential measured lamination} in a closed surface $F$ is a lamination $\displaystyle L=\bigcup_{j=0}^kL_j$ where $(L_0,L_1,\ldots, L_k)$  is finite sequence of measured laminations with $L_i$ embedded in $\displaystyle F\setminus \bigcup_{j<i}L_j$, and such that $\displaystyle \bigcup_{j\le i}L_j$ is an essential lamination in $F$ for each $i\le k$.
\end{defn}

To understand these laminations, we can begin by understanding the {\it level $i$ lamination $L_i$}.  This lamination $L_i$ is embedded in a non-compact surface
$\displaystyle\hat  F_i=F\setminus \bigcup_{j<i}L_j$.   We let $F_i$ denote the possibly non-compact completion of $\hat F_i$. We show a typical component of $S_i$ in Figure \ref{PairSurfFig1}(a), a surface with infinite outward cusps on its boundary.  It is convenient to truncate the outward boundary cusps as shown in Figure \ref{PairSurfFig1}(b) to obtain a compact surface pair $(\bar F_i,\alpha)$.  Some ends of the measured lamination $L_i$ disappear into the cusps, so when we truncate $F_i$ we obtain an essential measured lamination with boundary in the arcs $\alpha$ of truncation, which are bold in the figure.  Of course, we choose the truncation arcs to be efficient with respect to the lamination.

\begin{figure}[H]
\centering
\scalebox{1}{\includegraphics{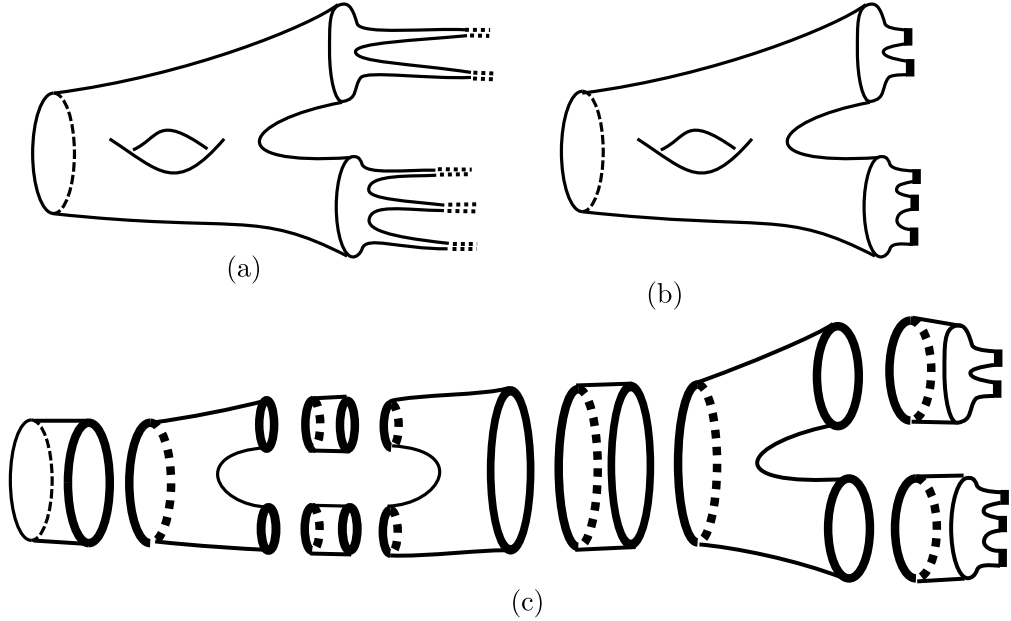}}
\caption{\small A component of $F_i$, its truncation, and its decomposition.}
\label{PairSurfFig1}
\end{figure}

In order to analyze the space of essential measured laminations in $F_i$, we shall use an extension of the method of Allen Hatcher, used in  \cite{AH:SurfaceMLS} and other sources.
This amounts to decomposing the surface into pairs of pants, then tightening the lamination $L_i$ to obtain measured laminations in each pair of pants, then 
reassembling the surface.  An important subtlety here is that when reassembling, one usually needs to reintroduce some ``twist" where two pairs of pants are glued on 
a curve.  Hatcher uses a ``connector annulus" to contain this twist.  The connectors can be seen in the decomposition of $F_i$ shown in  Figure \ref{PairSurfFig1}(c)
as annuli with both boundaries bold.   At the left side of  Figure \ref{PairSurfFig1}(c) we see another type of annular surface in the decomposition, which has one bold boundary curve and one ordinary boundary component.  The lamination $L_i$ may have leaves spiraling toward the left boundary of Figure \ref{PairSurfFig1}(a), so we will use the left annulus of the decomposition of Figure \ref{PairSurfFig1}(c) to contain this spiraling behavior.  

More formally, we are decomposing the truncated $(\bar F_i,\alpha)$ into {\it elementary surface pairs} of the form $(\bar F,\beta)$ where $\bar F$ is a compact surface
with boundary and $\beta$ is a closed submanifold of the boundary consisting of finitely many arcs and closed curves.    We will call $(\bar F,\beta)$ a {\it surface pair}.  

For an
arbitrary orientable surface pair $(\bar S,\alpha)$, it is reasonable to think of each closed curve of $\alpha$ as being obtained by truncating an infinite cusp of a surface $S$ with cusps, or as coming from a decomposition as above.  Our analysis will apply 
to arbitrary orientable surface pairs, which can be interpreted as surfaces with infinite cusps and infinite outward boundary cusps.   
We repeat that to each surface with cusps $S$ we associate by truncation a surface pair $(\bar S,\alpha)$ and vice versa.

Our discussion of finite depth measured laminations was intended only to motivate the study of measured laminations in surface pairs or surfaces with cusps, which is the topic of this paper.   Our goal is to calculate a projective space of measured laminations in a given surface pair $(\bar S,\alpha)$.  This extends the well-known theory due mostly to William Thurston.  In fact, most of the ideas in this paper exist in the literature.  

\begin{defn} If $(\bar S,\alpha)$ is a surface pair, the {\it geometric Euler characteristic} of the pair is defined as $\Chi_g(\bar S,\alpha)=\Chi(\bar S)-\frac 12 c$ where $\Chi$ denotes the usual Euler characteristic
and $c$ is the number of arcs in $\alpha$.   
\end{defn}

\begin{figure}[H]
\centering
\scalebox{1}{\includegraphics{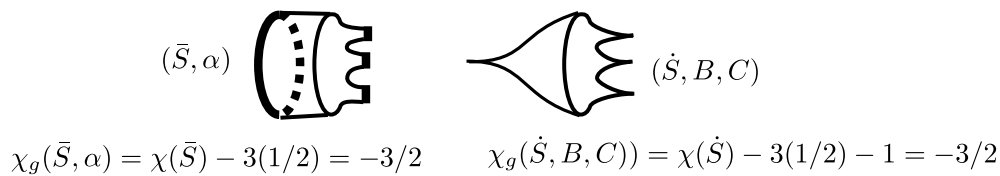}}
\caption{\small The geometric Euler characteristic.}
\label{PairGEulerFig}
\end{figure}

This Euler characteristic can equally well be defined for surfaces with finitely many {\it finite boundary cusps} and finitely many {\it finite cusps}.  Each finite boundary cusp is obtained from an arc of $\alpha$ by collapsing it to a point.  Each finite cusp is obtain from a closed curve of $\alpha$ by collapsing it to a point.  The boundary cusps are imagined as boundary points with zero interior angle and the finite cusps are imagined as cone points, with angle zero, see Figure \ref{PairGEulerFig}.    Thus, a surface pair $(\bar S,\alpha)$ is replaced by a surface $\dot S$ with a finite set $B$ of interior cusps and a set $C$ of boundary cusps.  We can write the surface with finite cusps as a triple $(\dot S,B,C)$ with cusps as shown in the figure.  If $F$ has $c=|C|$ boundary cusps and $b=|B|$ cusps, then $\Chi_g(\dot S,B,C)=\Chi(\dot S)-\frac 12 c -b$.  Here $\Chi(\dot S)$ is the Euler characteristic of the underlying surface $\dot S$.  We now have three versions of our objects of study:  A surface with cusps $S$, a surface pair $(\bar S,\alpha)$ obtained from $S$ by truncation, and a surface with finite cusps $(\dot S,B,C)$.  

By convention a surface pair $(D,\alpha)$ where $D$ is a disk and $\alpha$ consists of $n$ arcs is called an {\it n-gon}, $n=0,1,2,3,4\ldots$, but a 1-gon is also known as a {\it monogon}, while a 2-gon is also known as a {\it digon}.  Instead of ``digon" some authors use ``bigon,"  an awkward mix of Latin and Greek.   Here the model $(\dot D,C)$ with finite boundary cusps $C$ may be more suggestive.

\begin{defns} An {\it embedding of a train track $\tau$ in a surface pair $(\bar S,\alpha)$} is an embedding such that $\tau$ meets $\alpha$ transversely and $\tau$ contains any component of $\delta=\cl(\bdry\bar S-
\alpha)$ intersected by $\tau$.    Corresponding to an embedding of $\tau$ in a surface pair $(\bar S,\alpha)$ we have a {\it fibered neighborhood} of $N(\tau)$ as shown in Figure \ref{PairTrainFig}(a), with the frontier of $N(\tau)$ being the union 
of the horizontal boundary $\bdry_h(\bar\tau)$ and the vertical boundary $\bdry_v(\bar\tau)$ as shown.   The closure of the complement of $N(\tau)$ naturally has the structure of a surface pair $(\bar F,\beta)$ where $\beta$ consists of arcs in the vertical boundary of $N(\bar\tau)$, together with arcs of $\alpha\setminus \intr(N(\tau))$.   There is a projection map $\pi:N(\tau)\to \tau$ which collapses interval fibers to points.   Figure \ref{PairSpiralTrainFig} shows $\tau$ and $N(\tau)$ for a train track carrying a spiral leaf approaching a closed leaf in $\delta$, a key example.   

A lamination $L$ is {\it carried} by $\tau$ if it can be isotoped into $N(\tau)$ so it is transverse to the interval fibers of $N(\tau)$.   We further require that if $\epsilon$ is a component of $\delta$, there is exactly one leaf of $L$ contained in $\pi\inverse(\epsilon)$, and that leaf is (isotopic to) $\epsilon$.  The lamination $L$ is  {\it fully carried} by $\tau$ if in addition it intersects every fiber of $N(\tau)$.   The train track $\tau$ is {\it full} if $\tau$ fully carries some lamination.  \end{defns}

\begin{figure}[H]
\centering
\scalebox{1}{\includegraphics{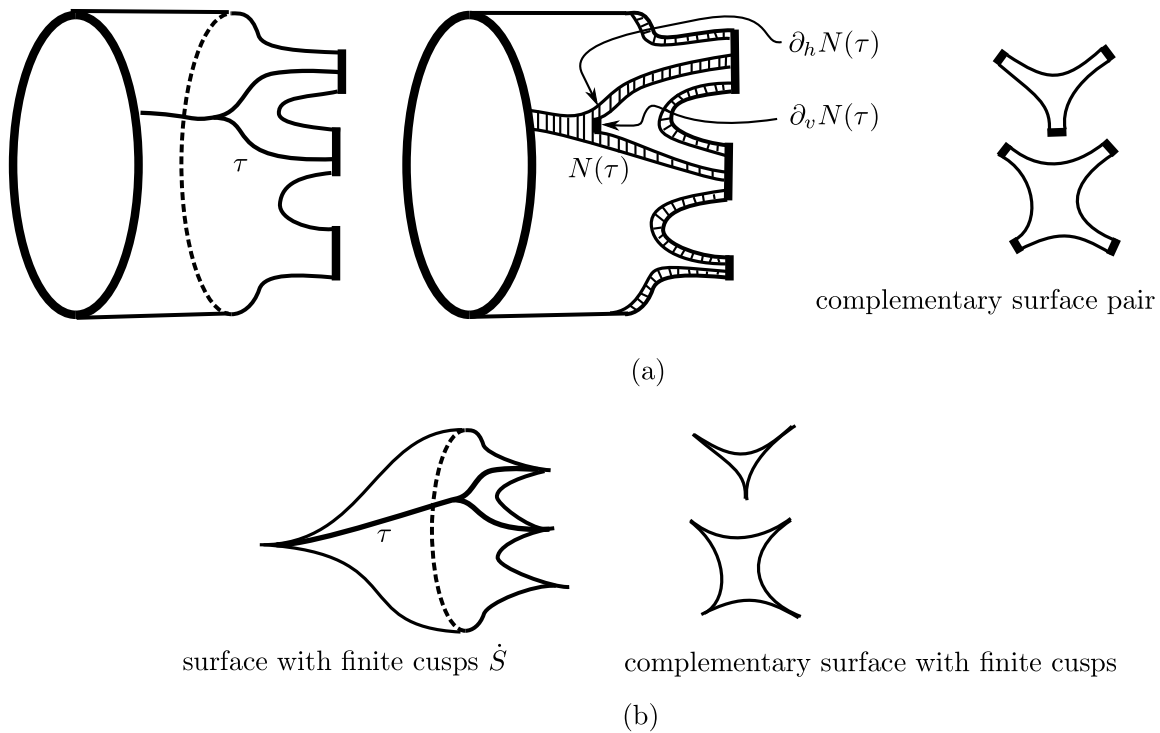}}
\caption{\small Train track in surface pair with complementary surface pair.}
\label{PairTrainFig}
\end{figure}

 The kind of train track shown in Figure \ref{PairSpiralTrainFig} will be important to us when we wish to allow leaves spiraling towards and limiting on a closed curve of $\delta$.

\begin{figure}[H]
\centering
\scalebox{1}{\includegraphics{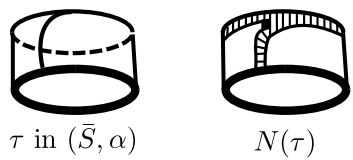}}
\caption{\small Train track for spiral.}
\label{PairSpiralTrainFig}
\end{figure}

Thanks to the above definition, we now have a  geometric Euler characteristic for surfaces in the complement of a train track neighborhood for an embedding $(\tau,\bdry\tau)\embed (\bar S,\alpha)$, as shown in the figure.    Applying the projection $\pi:N(\tau)\to \tau$ and collapsing each component of $\alpha$ to a cusp to replace $(\bar S,\alpha)$ by a cusped surface $(\dot S,B,C)$, the completions of the components of the complement of the train track are surfaces with cusps, whose $\Chi_g$'s we can calculate from our finite cusp point of view, see  Figure \ref{PairTrainFig}(b).

\begin{defns}  A {\it good train track} in $(\bar S,\alpha)$ is a train track $(\tau,\bdry\tau)\embed (\bar S,\alpha)$ with the property that every component of the complementary surface pair has strictly negative $\Chi_g$.  

 If $(\bar S,\alpha)$ is a surface pair with $\Chi_g(\bar S,\alpha)<0$, a lamination $(L,\bdry L)\embed (\bar S,\alpha)$ is {\it essential} if it is carried by a good train track $(\tau,\bdry \tau)$.
 
In general, for $\Chi_g(\bar S,\alpha)\le 0$, a {\it fair train track} is a train track $(\tau,\bdry\tau)\embed (\bar S,\alpha)$ with the property that every component the surface pair complementary to $N(\tau)$ has $\Chi_g\le 0$.

A {\it 2-dimensional Reeb train track (half Reeb train track)  in $\tau$} is an embedding of the train train track $\rho$ ($\sigma$), see Figure \ref{PairReebTTFig}(a)(b), in $N(\tau)$ transverse to fibers such that the two  smooth closed curves in the train track $\rho$ together bound an annulus in $S$ (the one closed curve in the train track $\sigma$ cuts from $(\bar S,\alpha)$ an annulus containing one closed curve of $\alpha$.
We also show in Figure \ref{PairReebTTFig}(c) the half Reeb train track in $(\bar S,\alpha)$ viewed in a surface with finite cusps.

A fair train track which contains no 2-dimensional Reeb train track or half Reeb train track is called an {\it essential train track}.    
 
 If $(\bar S,\alpha)$ is a surface pair with  $\Chi_g(\bar S,\alpha)=0$ a lamination $(L,\bdry L)\embed (\bar S,\alpha)$ is {\it essential} if it is fully carried by an essential train track.

An {\it essential measured lamination} in $(\bar S,\alpha)$ is an essential lamination $L$ in $(\bar S,\alpha)$ such that $L\setminus \delta$ has a transverse measure $\mu$ of full support.  This means the measure is positive on any transversal which intersects $L$.  We will sometimes abuse our definitions by referring to $L\setminus \delta$ as a measured lamination in $(\bar S,\alpha)$ although it may not be closed, hence not a lamination and also not measured in the usual sense where it spirals towards $\delta$.  It {\it is} measured in the usual sense when viewed as a lamination in $\hat S$.   

If $L$ is measured according to the definition above, with measure $\mu$, the lamination $L\cup \delta$ is also measured in a certain sense.  A {\it geometric measure} on $L\cup \delta$  corresponding to the measure $\mu$ is a transverse measure $\nu$ which assigns a non-negative element of the extended reals $\bar\reals$ to each transversal $T$ for the lamination.   If $T\cap \delta=\emptyset$, $\nu(T)=\mu(T)<\infty$.  If $T\cap \delta\ne \emptyset$ then $\nu(T)=\infty$.  
\end{defns}

The reason for the name ``geometric measure" is explained in \cite{UO:Depth}:  The geometric measure on $L\cup \delta$ yields a measured lamination in a surface ``dual to $S$," which approximates a hyperbolic structure on the dual surface.

\begin{figure}[H]
\centering
\scalebox{1}{\includegraphics{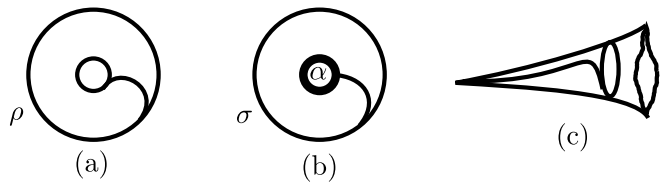}}
\caption{\small Reeb train tracks.}
\label{PairReebTTFig}
\end{figure}

We observe that the train track shown in Figure \ref{PairSpiralTrainFig} is not good, only fair, since its complementary surface is a digon, which has $\Chi_g=0$.  In surface pairs $(\bar S,\alpha)$ such that $\Chi_g(\bar S,\alpha)=0$, no essential lamination is carried by a good train track, which is why we must work with fair and essential train tracks.

We can now return to the strategy for understanding measured laminations in a surface pair $(\bar S,\alpha)$.    Since there is a one-one correspondence between measured laminations of the form $L\setminus\delta$ and the geometric measured laminations $L\cup \delta$, we will both points of view, whichever is most convenient.  The laminations $L$ $\hat S$ can have leaves that spiral to a closed curve in $\delta$.   We wish to define and describe a projective measured lamination space $\PM(\bar S,\alpha)$ for the surface pair $(\bar S,\alpha)$ whose points are projective classes of measured laminations $L$.   The strategy, as we mentioned at the beginning of the introduction, is to choose a decomposition $\D$ of the surface pair.  The decomposition cuts $(\bar S,\alpha)$ into a finite number of of elementary surface pairs, including pairs of pants and topological annuli of various kinds.   We will begin in Section \ref{Complexes} by determining the unprojectivized and projectivized spaces $\M(\bar S,\alpha)$ and $\PM(\bar S,\alpha)$ for the elementary surface pairs.  In many cases, the spaces are actually curve complexes $\PC(\bar S,\alpha)$.  Some of the elementary surface pairs have $\Chi_g=0$.   For arbitrary surface pairs $(\bar S,\alpha)$ with $\Chi_g(\bar S,\alpha)<0$, we will begin by describing in Section \ref{DecompositionPM} a space $\PM_\D(\bar S,\alpha)$ which depends on the decomposition, then in Section \ref{IntersectionNumbers} we will show that 
$\PM_\D(\bar S,\alpha)$ is homeomorphic to a space which does not depend on the decomposition $\D$.   If $\Chi_g(\bar S,\alpha)<0$, we consider a set $\H$ homotopy classes of arcs and closed curves in $(\bar S,\alpha)$, including oriented closed curves of $\delta$, but not including arcs of $\delta$.  We will embed the unprojectivized space $\M_\D(\bar S,\alpha)$ in $\reals^\H$.  Thus we will prove:

\begin{thm}  \label{MainThm} Suppose $(\bar S,\alpha)$ is a connected surface pair satisfying $\Chi_g=\Chi_g(\bar S,\alpha)<0$, with topological Euler characteristic $\Chi(\bar S)=\Chi$.  Suppose $\alpha$ contains $b$ closed curves and $c$ arcs.
Then $\M(\bar S,\alpha)$ is homeomorphic (via a homeomorphism linear on projective equivalence classes) to $\reals^{-3\Chi-b+c}\times \reals_+^b=\reals^{-3\Chi_g-c/2-b}\times \reals_+^b$, where $\reals_+$ denotes $[0,\infty)\subset \bar \reals$.  Thus $\PM(\bar S,\alpha)$ is homeomorphic to the join of a sphere $S^{-3\Chi-b+c-1}=S^{-3\Chi_g-c/2-b-1}$ and a simplex $\Delta^{b-1}$.
\end{thm}

Note that although we use extended reals to define our space, the calculation yields $\M(\bar S,\alpha)$ in terms of products of non-extended real lines.

The above theorem is a generalization of Proposition 1.5 of \cite{AH:SurfaceMLS}.   The proposition in Hatcher's paper is stated as a result, in our notation, about $\PM_\D(\bar S,\alpha)$, but later in the paper it is shown that
the space is independent of the choice of $\D$.  Hatcher's result, like ours, applies only to $(\bar S,\alpha)$ with $\Chi_g(\bar S,\alpha)$ strictly negative.  Unlike ours, it applies only to surface pairs with $\bdry 
S=\alpha$, but it includes the case of non-orientable surface pairs.

Now we explain the consequences of Theorem \ref{MainThm} for finite depth measured laminations in a surface $S$.   For simplicity, we will assume that $S$ is a closed surface, but one could formulate similar results for surface pairs.  Let $\FDM(S)$ denote the {\it set} of finite depth essential measured laminations in $S$.  We  describe a topology for this set in \cite{UO:Depth}.  Let $\FDM_k(S)$ denote the subset of laminations of  depth $\le k$.   Then there is a projection map $\Pi:\FDM_k(S)\to \FDM_{k-1}(S)$, which deletes the level $k$ measured lamination $L_k$ in the definition of a finite depth measured lamination.

\begin{corollary}  Let $ S$ be a closed surface and let $\displaystyle L=\bigcup_{j=0}^{k-1}L_j\in \FDM_{k-1}(S)$ and suppose the completion of $S\setminus L$ can be represented as a surface pair $ (\bar F,\beta)$.  Then the preimage $\Pi\inverse(\{L\})$ can be identified with $\M(\bar F,\beta)$.
\end{corollary}

I worked with two undergraduate students, Mohammed Iddrisu and Sharwri\\ Phutane on understanding the curve complexes in Section \ref{Complexes}.  I explained many of the ideas related to measured laminations to them.  The interaction with these students was very helpful to me.  I also thank Allen Hatcher, who pointed me to some known results.

\section{Curve complexes for elementary surface pairs.}\label{Complexes}  In order to determine projective measured lamination spaces of
surface pairs, we will need an understanding of projective measured lamination spaces of  certain
{\it elementary surface pairs}, which will be used to deal with the general case.   For most of these 
elementary surface pairs, the projective measured lamination space is the same as the projective curve complex.
We recall that our laminations in a surface pair $(\bar S,\alpha)$ can include the components of $\delta =\cl(\bdry \bar S-\alpha)$ but we can ignore these and consider measured ``laminations" disjoint from $\delta$.  Then an {\it essential curve system} is a system $C$ of disjointly embedded
arcs and closed curves such that $C$ with no curve isotopic to an arc or closed curve isotopic to an arc or closed curve in $\alpha$ or $\delta$.
Assigning weights to the curves of $C$, we obtain a measured lamination.  Adding $\delta$ to $C$ with atomic $\infty$ transverse measure on $\delta$, we obtain a lamination with geometric transverse measure.

We let $\C(\bar S,\alpha)$ denote the set of curve systems of essential curves.   If $\Chi_g(\bar S,\alpha)<0$ we can work with good train tracks;  interpreting $C\cup \delta$ as a train track,
the train track is good if complementary surface pairs have negative $\Chi_g$.  In any case, whether $\Chi_g(\bar S,\alpha)=0$ or $\Chi_g(\bar S,\alpha)<0$,
$C$ cannot contain any arcs or closed curves isotopic to components of $\delta$ or $\alpha$.

\begin{defn}  The space $\WC(\bar S,\alpha)$ is the space of weighted essential curve systems in $(\bar S,\alpha)$.  For every collection of
$k$ disjointly embedded essential curves (arcs or closed curves none of which is isotopic to a component of $\delta$) in $(\bar S,\alpha)$, a {\it curve system} $C$ say, the space contains the cone on a  $k-1$-simplex, i.e. points in the
first orthant  corresponding to weights $x_i\ge0 $ on each of the $k$ curves in $C$.  Two of these cones corresponding to curve systems $C_1$ and $C_2$ are identified on a sub-cone corresponding to 
the curve systems $C_3$ consisting of curves common to $C_1$ and $C_2$ (if any).  If we projectivize this space, we obtain the {\it curve complex} for $(\bar S,\alpha)$,
which we denote $\PC(\bar S,\alpha)$.  The curve complex is assembled from $k-1$-simplices of different dimensions, each corresponding to an essential curve system consisting of $k$ disjointly
embedded curves in $(\bar S,\alpha)$.  Two of these simplices are identified on a face (which could equal one of the simplices) corresponding to the system of common curves (if any).
\end{defn}

\begin{defns} 
Suppose $(\tau,\bdry \tau)\embed (\bar S,\alpha)$ is a train track.  An {\it invariant weight vector} is a vector $\bar w$ assigning a weight $w_i\in \bar\reals$ to each segment of $ \tau$ such that all switch equations hold and such that the weights on segments of $\tau\cap \delta$ are $\infty$.   We let $w$ be the weights on segments of $\tau\setminus\delta$, still satisfying switch equations.  We say $w$ is also an {\it invariant weight vector on $\tau$} with the understanding that a weight vector without the over-bar is a weight vector assigning weights only to segments of $\tau\setminus\delta$.  We let $\V(\tau)$ denote the cone of invariant weight vectors $w$ assigning weights to segments of $\tau\setminus \delta$.  This is a cone in the first orthant of $\reals^k$.  $\PV(\tau)$ is the projectivized $\V(\tau)$.  It is a convex polyhedron in the standard $(k-1)$- simplex in $\reals^k$.   We let $\V_\rationals(\tau)$ denote the set of rational weights in the cone.
\end{defns}

\begin{proposition}  Given a train track $(\tau,\bdry \tau)\embed(\bar S,\alpha)$, an invariant weight vector $x$ for $\tau$ (not assigning weights to segments in $\delta$) uniquely determines a measured lamination $(L,\mu)$ in $\hat S=S\setminus \delta$ with the property that $L\cup \delta$ is a geometric measured lamination carried by $\tau\cup \delta$.\end{proposition}

\begin{proof}  We are claiming that the weight vector
$x$ which assigns a weight $x_i$ to each segment $\sigma_i$ of $\tau\setminus\delta$ such that switch equations are satisfied, determines a measured lamination in $S\setminus\delta$ which may have leaves spiraling to a closed curve of $\delta$ and limiting on that curve.   The elementary theory of train tracks, shows that $x$ determines a measured lamination in $N(\tau\setminus\delta)$, but we must show how to extend it to $\tau$, with possible spiraling at closed components of $\delta$.

Consider first an arc $\kappa$ of $\delta$ with an orientation induced from the orientation of $S$.  We will assign weights to the segments of a regular neighborhood of $\kappa$ in $\tau$.   Suppose there are $k$ segments of  that regular neighborhood of $\tau$ attached to $\kappa$, with weights $x_i$, $i=1,\ldots k$. Suppose the weight (not yet determined) is $x_0\ge 0$ at the segment containing the initial point of $\kappa$.  Then on subsequent segments of $\kappa$ the weights must be the partial sums of the entire sum
$$x_0+\sum_{i=1}^k \epsilon_i x_i,$$
\noindent including the first term.   Here $\epsilon_i=\pm1$ indicates the sense of branching, see Figure \ref{PairExtendFig}(a).   We choose $x_0$ so that all the partial sums are $\ge 0$ and at least one partial sum is equal to $0$.  Then the weights on $\kappa$ extend the measured lamination as required.  

\begin{figure}[H]
\centering
\scalebox{1}{\includegraphics{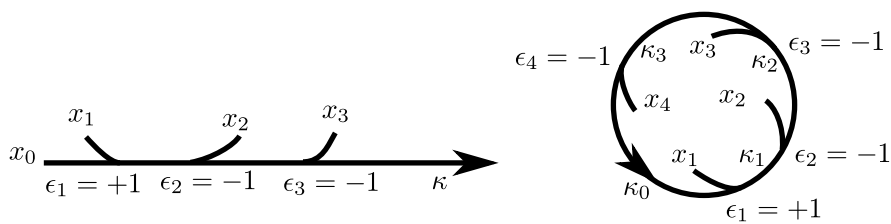}}
\caption{\small Extending the measured lamination near $\delta$.}
\label{PairExtendFig}
\end{figure}

Consider next a closed curve $\kappa$ of $\delta$.  Again, we orient $\kappa$ with orientation induced by the orientation of $S$.  Again suppose  there are $k$ segments of  the regular neighborhood of $\kappa$ in $\tau$ attached to $\kappa$, with weights $x_i$, $i=1,\ldots k$, where we order these segments and weights in cyclically increasing order with respect to the orientation.  See Figure \ref{PairExtendFig}(b).  Suppose $\kappa_0\subset \kappa$ is the segment preceding the switch where the segment with weight $x_1$ attaches to $\kappa$, and order the $k$ segments $\kappa_i$, $i=1,\ldots, k-1$ of $\kappa$ cyclically in increasing order with respect to orientation.  Consider the partial sums of 
$$\sum_{i=1}^k \epsilon_i x_i.$$
\noindent     If all of the partial sums are non-negative, then we assign the weight $0$ to the segment in $\kappa_0$ and we assign the $i$-th partial sum to $\kappa_i$, $i=1,\ldots, k-1$.   If some of the partial sums are negative, we relabel segments to start at a different point in the cycle, and possibly reverse the orientation of $\kappa$,  so that all of the partial sums are non-negative.  The weights on the segments of the regular neighborhood of $\kappa$ in $\tau$ do not in general define a measured lamination on the train track, because one switch equation (on the switch corresponding to $x_k$), may not be satisfied.   Instead we obtain a ``measured lamination with leaves spiraling towards $\kappa$" by pushing excess measure toward $\kappa$.   The set of switches  with weights $x_i$ and signs $\epsilon_i$ determine a cohomology class in $H^1(\kappa,\reals)$ which is unchanged when we perform appropriate splitting or pinching operations on the train track with weights.

From the weights on $\tau$, we have constructed a measured lamination in a regular neighborhood $N(\delta)\setminus \delta$ of $\delta$ in $\bar S$. The weights on $\tau\setminus N(\delta)$ determine a measured lamination on the remainder of $\bar S$, so combining these we obtain a measured lamination $(L,\mu)$ in $\hat S$.  Finally observe that if we replace on all segments in $\delta$ by $\infty$, then the weights on $\tau$ represent the union of $L\cup \delta$ with atomic infinite transverse measure on $\delta$, which yields the geometric transverse measure.
\end{proof}

\begin{defn} The notation $\tau(x)$ denotes the measured lamination $(L,\mu)$ in $\hat S$ constructed above.  We use $\tau(\bar x)$ to denote the geometric measured lamination determined by the invariant weight vector with $\infty$ entries on $\delta$.  Thus $\tau(\bar x)$ is the measured lamination $(L,\mu)$ union $\delta$ with infinite transverse measure on arcs with at least one endpoint in $\delta$.  \end{defn}

Using the previous proposition, we can show that if $(\bar S,\alpha)$ is a surface pair with the complement $\delta$ in  $\bdry S$ of $\alpha$ containing no closed curves, then any essential measured lamination $(L,\mu)$ is carried by a train track $\tau$ in $(\bar S,\alpha)$ such that every component of $\delta$ is also a component of $\tau$.   In other words, there is no interaction of the measured part of $L$ with $\delta$:

\begin{proposition}  An essential measured lamination $(L,\mu)$ in a surface pair $(\bar S,\alpha)$ is carried by a train track $\tau$ with no switches on any arc of $\delta$.
\end{proposition}

We will leave the proof of the proposition as an exercise.  It can be proved by splitting train tracks as in Definition \ref{SplitRespectDef}.

We can imitate the above definition of a curve complex in order to define some spaces which contain measured laminations as well as weighted curve systems.

\begin{defn} A collection $\T$ of (isotopy classes of) full train tracks embedded in $(\bar S,\alpha)$ is {\it closed} if it is closed under the operation of passing to full sub-train-tracks.
Corresponding to a closed collection of train tracks, we define $\PM_\T(\bar S,\alpha)$, a complex constructed as follows.  For every train track  $\tau\in\T$ we 
we include $\PV(\tau)$.  If $\tau_1$ and $\tau_2$ are train tracks sharing (up to isotopy) a common full sub-train-track $\tau_3$, which may be the same as $\tau_1$ or $\tau_2$, then we identify subcomplexes of  $\PV(\tau_1)$ and $\PV(\tau_2)$ corresponding to $\tau_3$.  We define $\M_\T(\bar S,\alpha)$ similarly.  

We say the collection $\T$ is {\it bijective} if every essential measured lamination is fully carried by exactly one train track $\tau$ in the collection, and there is only one weight vector on $\tau$ representing the lamination.
\end{defn}

Note that given any collection $\T$ of full train tracks in $(\bar S,\alpha)$, we can form a closed collection just by adding to the collection all full sub-train-tracks of the train tracks in the collection.  This property
is therefore not particularly interesting by itself.  Starting from a decomposition $\D$ of a surface pair we will produce a collection $\T$ of ``standard" train tracks which is not only closed, but also, as we shall eventually see, bijective.  We construct this system $\T$ by first constructing a similar system for each elementary surface pair coming from the decomposition $\D$.  At the same time, we will calculate measured lamination spaces for each of these elementary surface pairs.

The first cusped surface we will consider is $T_c$, which is a pair $T_c=(\bar S,\alpha)$, where $\bar S$ is an annulus, and
$\alpha$ consists of $c>0$ pairwise disjoint closed arcs in one boundary component union the other boundary component.  The notation comes
from the fact that we consider $T_c$ as a ``trim annulus" attached to a topological surface at a boundary component to obtain a surface pair with arcs in $\alpha$
in the corresponding boundary component.
The weighted arc systems in $T_c$ will be parametrized by the weights $x_1, x_2, \ldots, x_c$ induced on the arcs $\alpha_i$ of $\alpha$ together with
the weight $y$ induced on the closed curve of $\alpha$.

\begin{proposition}\label{CS1} The space $\PC(T_c)=\PC(\bar S,\alpha)$, $c>0$, embeds in the standard $c$-simplex $\Delta^c$ in $\reals^{c+1}$, with vertices corresponding to the 
parameters $x_1,x_2,\ldots, x_c,y$, as the union of the following subsets:

\begin{itemize}
\item The $(c-1)$-simplex $\{x_1+x_2+x_3+\cdots+ x_c=y\}$.

\item The subsets $\{x_i=0,\ x_1+x_2+x_3+\cdots+ x_c\ge y\}$, for $i=1,2\ldots c$.

\end{itemize}

\noindent The curve complex $\PC(T_c)$ is finite of dimension $c$ and can be expressed as a finite union of $c$-simplices.  For each simplex in the curve 
complex, we can choose a {\it elementary standard train track} in $T_c$ which fully carries all curve systems corresponding to interior points in the simplex, and which 
intersects each component of $\alpha$ in at most one point.  These elementary standard train tracks can be chosen to form a closed bijective collection $\T$ of train tracks in $T_c$.
\hip

\noindent The unprojectivized space $\C(T_c)$ is piecewise linearly homeomorphic to $\reals_+\times \reals^{c-1}$, with the parameter $y$ corresponding
to the $\reals_+=[0,\infty)\subset \reals$.  Also, $\C(T_c)$ can be identified with $\M_\T(T_c)$
\end{proposition}

\noindent Note:  The topological type of $\PC(T_c)$ is known, see \cite{AH:SurfaceTriangulations}.  We need a description of the space in terms of the parameters $x_i$ and $y$.

\begin{proof}  The proof is by induction.   For $c=1$, there is only one possible weighted arc system up to scalar multiplication,
namely the single arc $\rho_1$ shown in Figure \ref{PairTCFig1}.   We show $\PC(T_2)$ and $\PC(T_3)$ with their triangulations in Figure \ref{PairTCFig2}.
In Figure \ref{PairTCFig1}, we show examples in $T_1$, $T_2$, and $T_3$ of curve systems corresponding to top-dimensional cells in $\PC(T_c)$, $c=1,2,3$, as well as elementary standard train tracks that fully carry them.  Other elementary standard train tracks are constructed similarly, and there is considerable flexibility in choosing these train tracks.   Figure $\ref{PairTCFig2}$ also shows points in $\PC (T_c)$ corresponding to some arcs in $T_c$.

\begin{figure}[H]
\centering
\scalebox{1}{\includegraphics{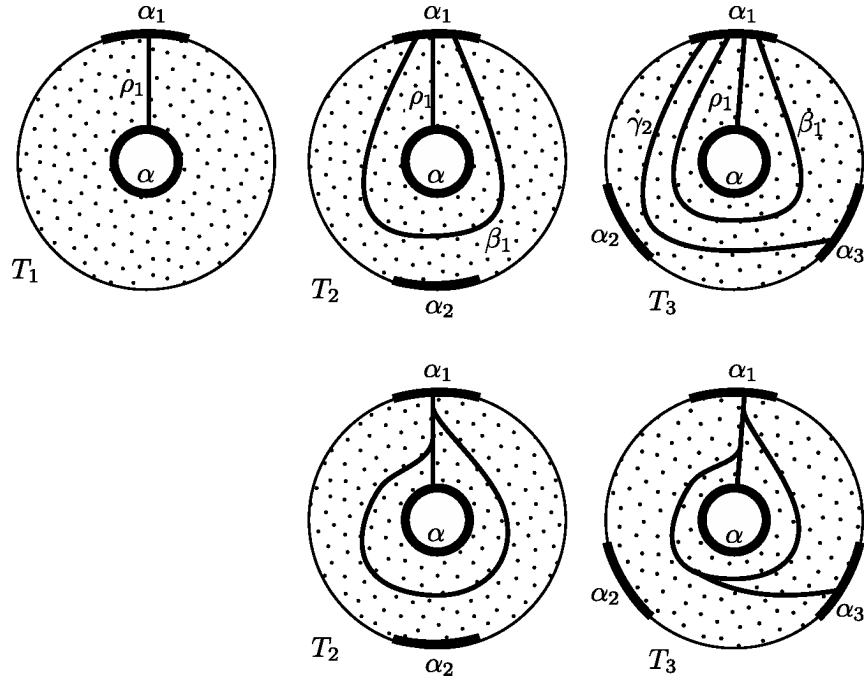}}
\caption{\small Examples of arc systems in $T_1$, $T_2$, and $T_3$ with standard elementary train tracks carrying them.}
\label{PairTCFig1}
\end{figure}

\begin{figure}[H]
\centering
\scalebox{1}{\includegraphics{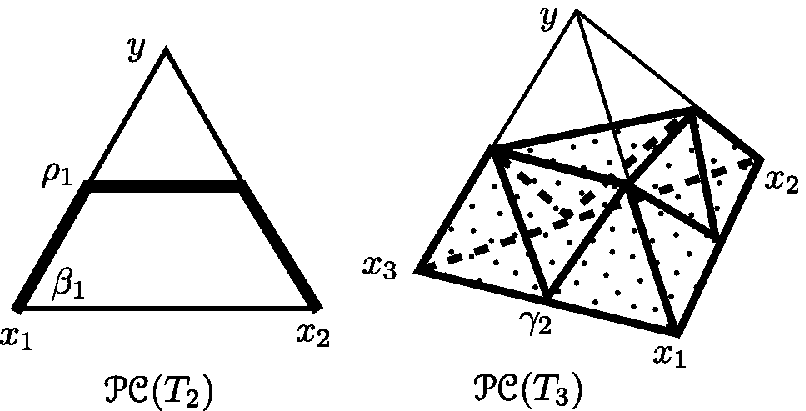}}
\caption{\small $\mathscr{PC}(T_2)$ and $\mathscr {PC}(T_3)$.}
\label{PairTCFig2}
\end{figure}

Now we prove the general statement using induction.  Assume $\PC(T_{c-1})$ is the union of subsets in the statement (with $c$ replaced by $c-1$).  
In the simplex $\Delta^c=[x_1,x_2,\ldots, x_c, y]$, putting $x_c=0$, we obtain the simplex  $\Delta^{c-1}=[x_1,x_2,\ldots, x_{c-1}, y]$ which contains $\PC(T_{c-1})$.  Every point
in $\PC(T_{c-1})$ contained in the face $x_c=0$ clearly also can represent a point of $\PC(T_c)$ whose induced weight $x_c$ is $0$.   There is an additional
arc $\gamma_c$ as shown in Figure \ref{PairTCFig3} which is essential in $T_c$ but inessential when we remove the arc $\alpha_c$ from $\bdry\bar S$, and 
which lies in the face $x_c=0$.  This arc can be added to any arc system in $T_c$ represented by a point in $\PC(T_{c-1})\subset [x_1,x_2,\ldots, x_{c-1}, y]$.
Hence coning $\PC(T_{c-1})\subset [x_1,x_2,\ldots, x_{c-1}, d]$ from the cone vertex $\gamma_c$ gives a subset of $\PC(T_c)\cap\{x_c=0\}$ which must in fact be all of $\PC(T_c)\cap\{x_c=0\}$ in the $c$-simplex.
The same reasoning gives the intersections of $\PC(T_c)$ with the other faces of the $c$-simplex, as required.

\begin{figure}[H]
\centering
\scalebox{1}{\includegraphics{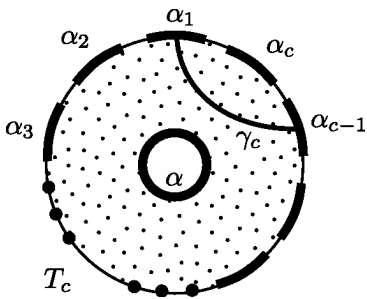}}
\caption{\small Remove $\alpha_c$ from $\alpha$ to get $T_{c-1}$ from $T_c$.}
\label{PairTCFig3}
\end{figure}

It remains to understand the portion of $\PC(T_{c})$ where all the $x_i$'s are non-zero.  But if all the $x_i$'s are non-zero, it is easy to verify that only ``radial" arc types $\rho_i$ can
appear.  This means arcs with one end in some $\alpha_i$ and the other end in the closed curve of $\alpha$, and in this case it is also easy to verify that 
$$\sum_{i=1}^c x_i=y.$$

For the last statement of the proposition, observe that the normal projection  in $\reals^{c+1}$ of $\bdry\C(T_c)$ to the hyperplane $$\displaystyle \sum_{i=1}^c x_i=0$$ \noindent is a piecewise linear homeomorphism (linear on rays through the origin) to the hyperplane.  The homeomorphism can be extended piecewise linearly to  $\C(T_c)$  preserving the positive $y$ coordinate axis.  \end{proof}

We must prove an analogue of Proposition \ref{CS1} for every elementary surface.  This has been done in many cases, see \cite{AH:SurfaceMLS}, and \cite{AH:SurfaceTriangulations}.
In particular, it has been done for the pair of pants, see \cite{AH:SurfaceMLS}.  We record the result below.   We use $P$ to denote the pair of pants $(P,\alpha)$ with $\alpha=\bdry P$.
For an essential weighted arc system in $P$ we let $y_1$, $y_2$ and $y_3$ denote the induced weights on the three boundary components of $P$.

\begin{proposition}\label{CS2} The space $\PC(P)=\PC(P,\bdry P)$ for the pair of pants $(P,\bdry P)$ is the standard $2$-simplex $\Delta^2$ in $\reals^3$, with vertices corresponding to the 
coordinate parameters $y_1,y_2,y_3$.  The curve complex is triangulated with four 2-simplices.  For each of the 2-simplices,  we can choose an elementary standard train track which carries
all arc systems represented by points in the simplex, see Figure \ref{PairPFig1}.  The train tracks can be chosen to form a closed bijective collection $\T$.
\hip\noindent
The unprojectivized space $\C(P)$ is piecewise linearly homeomorphic to $\reals_+^{3}$ and can be identified with $\M_\T(P)$.
\end{proposition}

\begin{figure}[H]
\centering
\scalebox{1}{\includegraphics{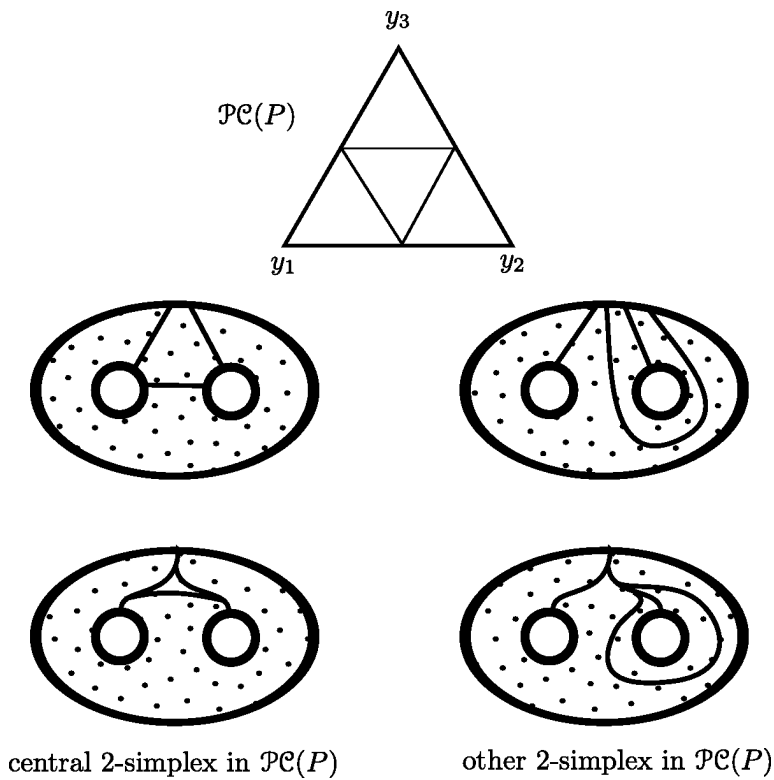}}
\caption{\small $\PC (P)$, typical arc systems, and standard elementary train tracks.}
\label{PairPFig1}
\end{figure}

The disk $D_c$ with $c>2$ cusps is an elementary surface which is self-sufficient.  We do not need it in the induction to paste together elementary surfaces.  For this reason, it is not
important to understand parameters for the curve complex $\C(D_c)$.    The following proposition is due to John Harer, \cite{JLH:VirtualCohomological}; an elementary proof can be found in \cite{AH:SurfaceTriangulations}.

\begin{proposition}\label{CS4} The projective space of curve systems $\PC(D_c)$ for the disk with $c$ cusps is homeomorphic to $S^{c-4}$.  
\end{proposition}

Finally, we need another elementary surface, called a {\it connector} in \cite{AH:SurfaceMLS}.  This is a pair $Q=(Q,\alpha)$ where $Q$ is an annulus and $\alpha=\bdry Q$.  The projective curve complex for such an annulus consists of a single point, so is not interesting, but connectors are used to incorporate ``twisting" between other surface elements, of the type $P$, $T_c$.  No connectors are needed adjacent to a surface element of the type $T_\emptyset$.  In the connector elementary surface, we examine $\PM_\T(Q)$, the projective measured lamination space of measured laminations carried by by a collection $\T$ of two prescribed standard train tracks, which are fair train tracks, not good train tracks.  The train tracks are not essential either, since they contain half Reef branched surfaces, which are needed to incorporate twisting.   These train tracks with their full sub-train-tracks form a closed collection $\T$ of train tracks.

\begin{proposition}\label{CS5} For the connector annulus we prescribe two train tracks $\tau_1$ and $\tau_2$, shown in Figure \ref{PairQFig} with parameters $(t_1, y_1)$ giving weights on $\tau_1$ and $(t_2, y_2)$ giving weights on $\tau_2$.   These two train tracks with their full sub-train-tracks form a closed collection $\T$ of train tracks.   The space $\PM_\T(Q)$  of projective measured laminations carried by these two train tracks is homeomorphic to $S^1$ .   The unprojectivized space $\M_\T(Q)$ of measured laminations carried by these two train tracks is piecewise linearly homeomorphic to $\reals^2$, with $y$ a piecewise linear function on this plane.
\end{proposition}

\begin{figure}[H]
\centering
\scalebox{1}{\includegraphics{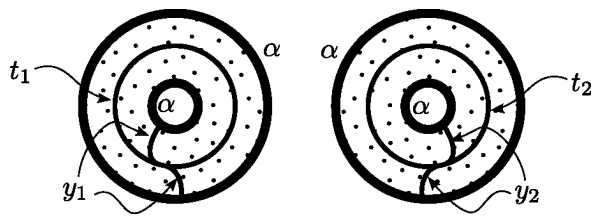}}
\caption{\small Elementary standard train tracks for the connector $Q$.}
\label{PairQFig}
\end{figure}

\begin{proof}  This is explained in \cite{AH:SurfaceMLS}, but we repeat the proof.  
The weights $(t_1,y_1)$ and $(t_2,y_2)$ for the two train tracks give first quadrants in two planes.
When $y_1=y_2=0$ the measured laminations obtained on each of the two train tracks are determined by
just one weigh $t_1$ or $t_2$, and in fact if also $t_1=t_2$, we obtain the same weighted closed curve.  Thus
we identify the positive $t_1$ axis with the positive $t_2$ axis.   Similarly, when $t_1=t_2=0$, and $y_1=y_2$, 
the weights on the two train tracks also represent the same weighted arc, so we identify the positive $y_i$-axes.  These 
identifications can be achieved using piecewise  linear homeomorphisms to a plane $\reals^2$ taking quadrants to half-planes.
If we let $y$ denote the weight induced by a curve system on either of the boundary components of $Q$, then $y$ is a piecewise linear function on the plane $\M_\T(Q)=\reals^2$, since either $y=y_1$ or $y=y_2$.  
\end{proof}

Among the elementary surfaces, we also have the trim annulus $T_\emptyset$.   This can be regarded as a $T_c$ with $c=0$, but it must be analyzed separately.  It is an annulus $(\bar S,\alpha)$ where $\alpha$ is one of the components of $\bdry\bar S$ and $\delta$ is the other boundary component.  There are no
compact essential arcs in this elementary surface, but we must allow half-infinite curves homeomorphic to $[0,\infty)\subset \reals$.  There are two such curves with boundary in the one component of $\alpha$, and with the end spiraling and limiting on $\delta$, see Figure \ref{PairTEFig}.  We record the obvious statement concerning this elementary surface in the following proposition.

\begin{figure}[H]
\centering
\scalebox{1}{\includegraphics{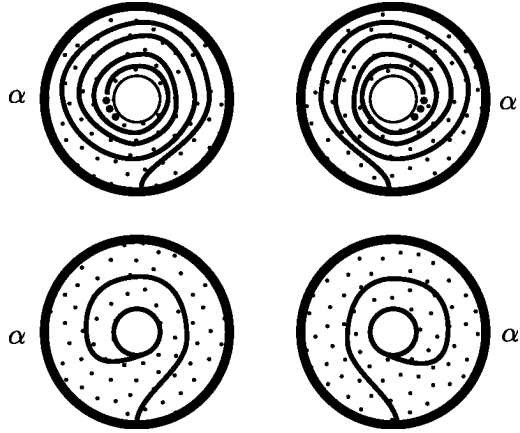}}
\caption{\small Curves and elementary standard train tracks for $T_\emptyset$.}
\label{PairTEFig}
\end{figure}

\begin{proposition}\label{CS3} The space $\PC(T_\emptyset)$ is $S^0$ with vertices corresponding to the coordinate $y_1$ and $y_2$, equal to the weight on the boundary, $y_1$ being the weight of a spiral in one sense and $y_2$ being the weight on a curve spiraling in the other sense.   Each of the two curves is carried by a  train track which includes $\delta=\cl(\bdry\bar S\setminus \alpha)$, see Figure \ref{PairTEFig}. 

\hip

\noindent The unprojectivized space $\C(T_\emptyset)$ is $\reals$.  The weight $y$ induced on the curve $\alpha$ is piecewise linear on $\reals$, namely it is the absolute value function.  \end{proposition}

Recall that measured laminations carried by the two train tracks in the above statement are represented by weights on segments of $\tau_i$ {\em excluding $\delta$}.

\section {$\PM_\D(\bar S,\alpha)$.} \label{DecompositionPM} 

In this section we suppose we are given a surface pair $(\bar S,\alpha)$ with a decomposition $\D$ of $(\bar S,\alpha)$ (as in the introduction) into 
elementary surface pairs of the following types described in the previous section:  

(1) Pairs of pants $(P,\bdry P)$ in which all of $\bdry P$ belongs to $\alpha$.

(2)  Connecter annuli $(Q,\bdry Q)$,.

(3)  Trim annuli $T_c$ of the form $(\bar S,\alpha)$ where $\bar S$ is an annulus and $\alpha$ consists of one component of $\bdry\bar S$ and $c$ arcs in the other boundary component.

(4)  Trim annuli $T_\emptyset$ of the form $(\bar S,\alpha)$ where $\bar S$ is an annulus and $\alpha$ consists of one component of $\bdry\bar S$.

In Section \ref{Complexes}, we described a collection of elementary standard train tracks in the elementary surface pairs arising from the decomposition $\D$.  By glueing elementary standard train tracks when reassembling $(\bar S,\alpha)$ from the elementary surface pairs of the decomposition, we obtain a collection $\T$ of standard train tracks relative to the decomposition.

\begin{defn}  A {\it standard train track} with respect to the decomposition $\D$  in $(\bar S,\alpha)$ is a train track $(\tau,\bdry \tau)$ such that for every elementary surface pair $(\bar F,\beta)$ of the decomposition, viewed as a subsurface of $(\bar S,\alpha)$, $\tau\cap (\bar F,\beta)$ is an standard elementary train track properly embedded in $(\bar F,\beta)$.
\end{defn}

We can easily verify the following.

\begin{lemma} The standard train tracks $\T$ associated to a decomposition $\D$ of $(\bar S,\alpha)$ form a closed collection of train tracks.
\end{lemma}

Eventually, we will also have to show that $\T$ is a bijective collection of train tracks.

\begin{defn}  We define a space $\PM_\D(\bar S,\alpha)$ to be the same as $\PM_\T(\bar S,\alpha)$, where $\T$ is the closed collection of standard train tracks with respect to $\D$ in $(\bar S,\alpha)$.
\end{defn}

The first step towards proving Theorem \ref{MainThm} is to prove $\PM_\D(\bar S,\alpha)$ has the topological type described in the theorem.  It then remains to show that the space does not depend on
the decomposition.

\begin{proposition}  Suppose $(\bar S,\alpha)$ is a connected surface pair satisfying $\Chi_g=\Chi_g(\bar S,\alpha)<0$, with topological Euler characteristic $\Chi(\bar S)=\Chi$.  Suppose $\alpha$ contains $b$ closed curves and $c$ arcs.
Then $\M_\D(\bar S)$ is homeomorphic (via a homeomorphism linear on projective equivalence classes) to $\reals^{-3\Chi-b+c}\times \reals_+^b=\reals^{-3\Chi_g-c/2-b}\times \reals_+^b$, where $\reals_+$ denotes $[0,\infty)\subset \reals$.  Thus $\PM_\D(\bar S,\alpha)$ is homeomorphic to the join of a sphere $S^{-3\Chi-b+c-1}=S^{-3\Chi_g-c/2-b-1}$ and a simplex $\Delta^{b-1}$.
\end{proposition}

\begin{proof}  Suppose the decomposition $\D$ of $(\bar S,\alpha)$ gives the following elementary surfaces:

(1) $k$ pairs of pants $P$.

(2) $\ell$ trim annuli $T_\emptyset$.

(3) Trim annuli $T_{c_1}, T_{c_2}, \ldots, T_{c_r}$ with $c_i$ arcs of $\alpha$ in $T_{c_i}$.  

(4) $m$ connectors $Q$.

Let $c=c_1+c_2+\cdots+c_r$.  We can easily calculate the projective lamination space of the disjoint union of the elementary surface pairs in the decomposition as a product of the spaces for elementary surface pairs.

For the $k$ pairs of pants we have $\reals_+^{3k}$.  

For the $\ell$ copies of $T_\emptyset$ we have $\reals^\ell$.  

For the disjoint union of $T_{c_1}, T_{c_2},\ldots, T_{c_r}$ we have

$(\reals_+\times \reals^{c_1-1})\times(\reals_+\times \reals^{c_2-1})\times\ldots\times(\reals_+\times \reals^{c_1-1})$.

For the $m$ copies of the connector $Q$ we have $\reals^{2m}$.

Next we note that whenever a closed curve of $\bdry T_{c_i}$ or a curve of $\bdry P$ is identified with a closed curve in the boundary of a connector or a $T_\emptyset$,
the $y\in\reals_+$ parameter for that boundary of the pair of pants or trim annulus $T_{c_i}$ is determined by the parameters for the attached connector or $T_\emptyset$, see Section \ref{Complexes}.  Thus
after identifications, we obtain a space homeomorphic to:

$$\reals_+^b\times \reals^\ell\times \reals^{(c_1+c_2+\cdots+c_r-r)}\times \reals^{2m}$$

\noindent  To make a connection with the Euler characteristic, notice that the total number of boundary curves of the pairs of pants in the decomposition is $3k$.   Since $\Chi=\Chi(S)=-k$, we can say that the total number of boundary curves of the pairs
of pants is $-3\Chi=3k$.   We know that $2m+\ell$ boundaries of connectors and $T_\emptyset$'s are attached to $3k+r-b$ boundaries of pairs of pants and $T_c$'s, so we have $2m+\ell=3k+r-b=-3\Chi+r-b$.    The number of $\reals$ factors in our space is $\ell+c-r+2m$, or replacing $2m+\ell$ by $-3\Chi+r-b$, the number is $-3\Chi+r-b+c-r=-3\Chi-b+c$, which gives the result in the statement.

To finish the proof, observe that our formula gives the correct answer for $n$-gons, $n\ge 4$, and also observe that our calculations apply to disconnected surface pairs.
\end{proof}

\section {Intersection numbers.} \label{IntersectionNumbers}

We will follow the usual strategy for defining reasonably a measured lamination space $\M(\bar S,\alpha)$ which for any given decomposition $\D$ of the surface pair $(\bar S,\alpha)$ is homeomorphic of $\M_\D(\bar S,\alpha)$.   There are two problems to address:

Problem 1:  Show that every essential measured lamination in $(\bar S,\alpha)$ is represented by a point of $\M_\D(\bar S,\alpha)$ and that different points of $\M_\D(\bar S,\alpha)$ do not represent the same measured lamination, up to isotopy.  If this were not true, $\M_\D(\bar S,\alpha)$ would not be a reasonable candidate for the measured lamination space.

Problem 2: Show that the topology of $\M_\D(\bar S,\alpha)$ does not depend on the choice of decomposition $\D$,

To solve Problem 2, we obtain a topology independent of choices by mapping each measured lamination to a point in $\reals^\H$, where $\H$ represents a set of homotopy classes of curves in $(\bar S,\alpha)$, including paths beginning and ending in $\alpha$ as well as closed curves in $S$.  More specifically, $\H$ will include curves of the following kinds:
\begin{enumerate}[(i)]\itemsep-3pt
\item Closed curves $\gamma$ not homotopic into $\alpha$ or into $\delta$ and not null homotopic.
\item Paths $\gamma$ beginning and ending in $\alpha$ and not homotopic into $\alpha$ or $\delta$.
\item Closed curves of $\delta$ with orientations.  If $\gamma$ is a closed curve in $\delta$, $\gamma_+$ is $\gamma$ with orientation induced from a given orientation of $\bar S$, while $\gamma_-$ is the same curve with opposite orientation.  The curves $\gamma_+$ and $\gamma_-$ are called {\it oriented variants of $\gamma$}.   $\H$ includes the two oriented variants for each closed component of $\delta$, but it does not include the unoriented curve.
\end{enumerate}

The elements of $\H$ are simply truncated geodesics in $S$, except closed geodesics in $\bdry S$ are given orientations.  Not all truncated geodesics are included;  geodesics spiraling to closed curves of $\delta$ are omitted.

For $\theta\in \H$ as above, we define an intersection or length function $i_\theta$ as follows:

\begin{enumerate}[(i)]\itemsep-3pt
\item If $\theta=\gamma$ is a closed unoriented curve not homotopic into $\delta$, this is the usual intersection in the standard theory of measured laminations, $i_\theta((L,\mu))=\inf\{\mu(\gamma)\}$ where the infimum is taken over curves homotopic to $\gamma$ and transverse to $L$.

\item For paths $\gamma$ beginning and ending in $\alpha$ the same formula applies, but the infimum is taken over curves $(\gamma,\bdry\gamma)\to (\bar S,\alpha)$ in the homotopy class of pairs and transverse to $L$.

\item For $\theta$ a closed path $\gamma$ in $\delta$ with a chosen orientation $\theta=\gamma_-$ or $\theta=\gamma_+$ and an essential measured lamination $(L,\mu)$, the intersection is the measure $w$ of leaves spiraling to $\delta$ if the sense of spiraling agrees with the orientation on $\gamma$ as shown in Figure \ref{PairSign}, with the lamination replaced by a train track, and is 0 otherwise.  
\end{enumerate}

\begin{figure}[H]
\centering
\scalebox{1}{\includegraphics{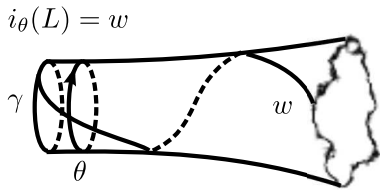}}
\caption{\small Intersections with oriented variants $\theta$ of a closed curve $\gamma$ in $\delta$.}
\label{PairSign}
\end{figure}

Intersections of closed oriented paths $\gamma$ in $\delta$ with an essential measured lamination $(L,\mu)$ can be interpreted as cohomology classes associated to a measured lamination with leaves spiraling to $\delta$. Suppose $(L,\mu)$ is carried by an extended train track $\tau\supset \delta$.  Then the measure $\mu$ gives an invariant weight vector on $\tau\setminus \delta$, and we can transversely orient switches of $\tau$ on $\delta$ as shown in Figure \ref{PairClassFig}, according to the sense of branching, and assign the weight from the invariant weight vector on $\tau$.   In this way $\bdry \tau$ with assigned weights represents a cohomology class in $H^1(\delta,\reals)$.     (The figure shows the intersection of the train track $\tau$ with a trim annulus $T_\emptyset$.)   In fact, intersections with elements of $\H$ determine the algebraic intersection of oriented closed curves $\gamma_+$ in  $\delta$ with $(L,\mu)$ as $i_{\gamma_+}(L)-i_{\gamma_-}(L)$, which gives the cohomology class.

\begin{figure}[H]
\centering
\scalebox{1}{\includegraphics{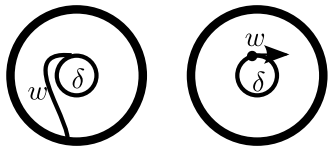}}
\caption{\small The cohomology class on $\delta$ induced by a measured lamination.}
\label{PairClassFig}
\end{figure}

We define a set-theoretic map $\I=(i_\gamma):\M_\D(\bar S,\alpha)\to \reals^\H$ with coordinate functions $i_\theta$, $\theta\in \H$.   We then assign the subspace topology to the image in $\reals^\H$ to obtain $\M(\bar S,\alpha)$.  Projectivizing in $\reals^\H$, we obtain the quotient which we call $\PM(\bar S,\alpha)$.   Our goal, of course, is to show that the map $\I=(i_\gamma):\PM_\D(\bar S,\alpha)\to \PM(\bar S,\alpha)$ is a homeomorphism.

We will begin by addressing Problem 1 after presenting a few necessary definitions and a lemma.

We quote the following lemma without proof.

\begin{lemma} Suppose $(\bar S,\alpha)$ is a surface pair and $L$ is an essential measured lamination in the pair.
Suppose $C$ is a curve system in $(\bar S,\alpha)$ consisting of disjointly embedded closed curves with no two of the closed curves isotopic and no curve isotopic to a curve of $\delta$.  Then $L$
can be isotoped to achieve the minimum intersection with $C$.  This means $L$ can be isotoped to achieve the intersection $\displaystyle \sum_\gamma i_\gamma(L)$ over connected curves $\gamma$ in $C$.  For a single curve $\gamma$ not homotopic into $\delta$ (not necessarily embedded) the infimum in the definition of $i_\gamma(L,\mu)$ is achieved by a representative of the homotopy class of $\gamma$.  
\end{lemma}

The lemma can be proved using hyperbolic geometry, by representing laminations and curve systems by geodesic laminations and geodesic curve systems.  Alternatively, the proof can be done directly as in \cite{AH:SurfaceMLS}, where it is shown (without using any geometry) that the infimum in the definition of $i_\gamma(L)$ can be realized.

\begin{proposition}\label{InjectiveProp}  If $\Chi_g(\bar S,\alpha)<0$ and $\D$ is a decomposition of $(\bar S,\alpha)$ then every essential measured lamination is realized as a point in $\M_\D(\bar S,\alpha)$.   The map $\M_\D(\bar S,\alpha)\to \reals^\H$ is an injection.
\end{proposition}

\begin{proof}  To prove the first statement, begin by choosing a hyperbolic structure for the surface with cusps $S$ associated to $(\bar S,\alpha)$   This could be done, for example, by doubling the non-compact surface, or by choosing an ideal triangulation for the surface.  Next, given an essential measured lamination $L$ in $(\bar S,\alpha)$, realize it as a geodesic lamination in $S$.  We are given a decomposition $\D$ associated to a  decomposition on curves of the curve system $C$ with components $C_i$.  The elementary surfaces for the decomposition associated to $C$ are:
\begin{enumerate}[(i)]\itemsep-4pt
\item If $\bar S$ is a disk and $C=\emptyset$, then $(\bar S,\alpha)$ is itself elementary, a disk with $c\ge 3$ cusps.
\item Pairs of pants $(P,\beta)$, where $\beta$ consists of  some union of boundary components of $P$ .
\item Trim annuli $(\bar F,\beta)$  where $\beta$ is a non-empty collection of arcs in one of the boundary components of $\bar F$ together with the other boundary component.  Each trim annulus is cut off by some $C_i$ in $C$ which is boundary parallel in $\bar S$ but not in the pair $(\bar S,\alpha)$.
\item Trim annuli $T_\emptyset=(\bar F,\beta)$, where $\beta$ is one boundary component of $\bar F$.
\item Connector annuli $(Q,\beta)$ associated to some curves $C_i$ of $C$, where $\beta$ is all of $\bdry Q$.   Each connector annulus $Q$ is a regular neighborhood of some $C_i$. (Recall there is no connector annulus adjacent to a $T_\emptyset$.)
\end{enumerate}

If we realize the curve system $C$ as a geodesic system, and the lamination $L$ as a geodesic lamination, then $L$ meets $C$ transversely, except any closed leaves of $L$ which are isotopic to a component $C_i$ of $C$.  If we incorporate ``twist" in connectors $Q_i$, we are left with essential weighted curve systems in pairs of pants $(P,\beta)$ of the decomposition and trim annuli $(T,\beta)$ of the decomposition.  There is something to check here, namely that the intersection of $L$ with each elementary surface is actually essential, but we leave this to the reader.   In this way we see that $L$ is represented by a point of $\M_\D(\bar S,\alpha)$.  

The next task is to show that the points of $\M_\D(\bar S,\alpha)$ are distinguished by intersection numbers $i_\gamma(L)$.  Here we extend the arguments in \cite{AH:SurfaceMLS}.  We begin with the special case that $(\bar S,\alpha)$ is the disk with $c$ arcs on its boundary.  We dealt with this surface pair in Proposition \ref{CS4}.  Clearly the measured laminations (which are weighted curve systems) in such a disk are determined by the weights at the cusps.  The weight at a cusp or arc $\alpha_1$ is detected by the intersection number $i_\gamma(L)$ where $\gamma$ is an arc joining the two arcs adjacent to $\alpha_1$, say $\alpha_2$ and $\alpha_c$.  (Recall that there are no essential weighted curve systems if $c\le 3$.)  This proves the injectivity of the map $\I:\M_\D(\bar S,\alpha)\to \reals^\H$ in this case.

Suppose $(\bar S,\alpha)$ is a pair satisfying $\Chi_g(\bar S,\alpha)<0$ whose underlying space is not a disk.   The decomposition into elementary surfaces is non-trivial, corresponding to a non-empty essential curve system $C$.   We will find a finite number of $\theta\in \H$ such that the intersection functions $i_\theta$ are sufficient to distinguish all measured laminations in $(\bar S,\alpha)$.  We will do this by showing that the parameters we use for the intersections of a measured lamination $L$ with each elementary surface are determined by the $i_\gamma$'s we consider.

We begin with curves $\gamma$ equal to the $C_i$ used for the decomposition.   These immediately give the parameters $y_i$ for any pair of pants $(P,\bdry P)$ in the decomposition $\D$ as well as parameters $y_i$ assigning weights to closed curves in the $\alpha$-boundary of trim annuli $T_c$.
Now we add essential arcs $\gamma\in \H$ to our finite collection which join arcs of $\alpha$ in the same boundary component of $S$ and which are boundary parallel in $S$.  In particular, we need such arcs $\gamma$ connecting $\alpha_i$ and $\alpha_{i+2}$ mod $c$ if there are $c$ arcs $\alpha_i$ of $\alpha$ on a boundary component of $S$, labeled cyclically.  In addition, if a boundary component of $S$ contains two arcs of $\alpha$, we need essential boundary parallel arcs joining each arc of $\alpha$ to itself.  Using intersections with all of these curves $\gamma$, it is easy to check that the parameters $x_i$ for trim annuli $(\bar F,\beta)$ are determined by these $i_\gamma$'s.  The case $\beta=\emptyset$, however, is a special case.  In this case, we take $\gamma$ to be the curve of $\delta$ in $\bdry\bar S$, and use oriented variants $\theta=\gamma_+$ and $\theta=\gamma_-$ .  The intersection numbers $i_\theta$ then determines the weight of the spiral lamination approaching $\delta$.  These intersection numbers also determine the sense of spiraling.

Any subtlety in our argument comes from the need to determine the twist parameters $t_i$ associated to a connector $(Q,\bdry Q)$in the decomposition $\D$.   Suppose the connector $Q$ is a regular neighborhood of the curve $C_j$ in $C$. There are four cases according to whether (a) the connector joins two pairs of pants, (b) joins one boundary component of a pair of pants $P$ to another boundary component, (c) joins two trim annuli $T_c$ with $c\ge 0$, or (d) joins one trim annulus $T_c$ ($c>0$) to a pair of pants $(P,\bdry P)$.  In each case a curve $\gamma$ shown in Figure \ref{PairTwistFig} clearly determines the weight $t_1$ or $t_2$ in the standard train track in $Q$.   A second curve $\gamma'$ obtained from $\gamma$ by applying a Dehn twist on $C_j$ will be needed to resolve the ambiguity between positive or negative twist in the connector.  To prove that intersection numbers with these curves $\gamma$ and $\gamma'$ determine the twist parameter, one needs to do a case-by-case analysis.  In the case where the connector joins a trim annulus containing two arcs of $\alpha$ to a pair of pants, we show in Figure \ref{PairTwist2Fig} the arcs $\gamma$ and $\gamma'$ homotoped to be efficient with respect to standard train tracks in the combined surface pair.    See the next paragraph for a definition of ``efficient," and the method for calculating intersection numbers.  We obtain different intersection numbers depending on the sense of twisting.  For one sense of twisting we get $i_\gamma((L,\mu))-i_{\gamma'}((L,\mu))=-2y_3-2y_4$, for the other sense of twisting we get $i_\gamma((L,\mu))-i_{\gamma'}((L,\mu))=y_3$.  Also, the formulas in the figure show how to calculate the twist parameters $t_1$ and $t_2$  given the other parameters together with $i_\gamma(L)$ and $i_{\gamma'}(L)$.   See Figure \ref{PairQFig} for parameters $t_1,t_2$, which both appear as $t$ in Figure \ref{PairTwist2Fig}. To complete the proof, one must check that similar calculations are possible when the trim annulus contains a different number of arcs of $\alpha$, and one must check that similar calculations are possible in the remaining three cases.

\begin{figure}[H]
\centering
\scalebox{1}{\includegraphics{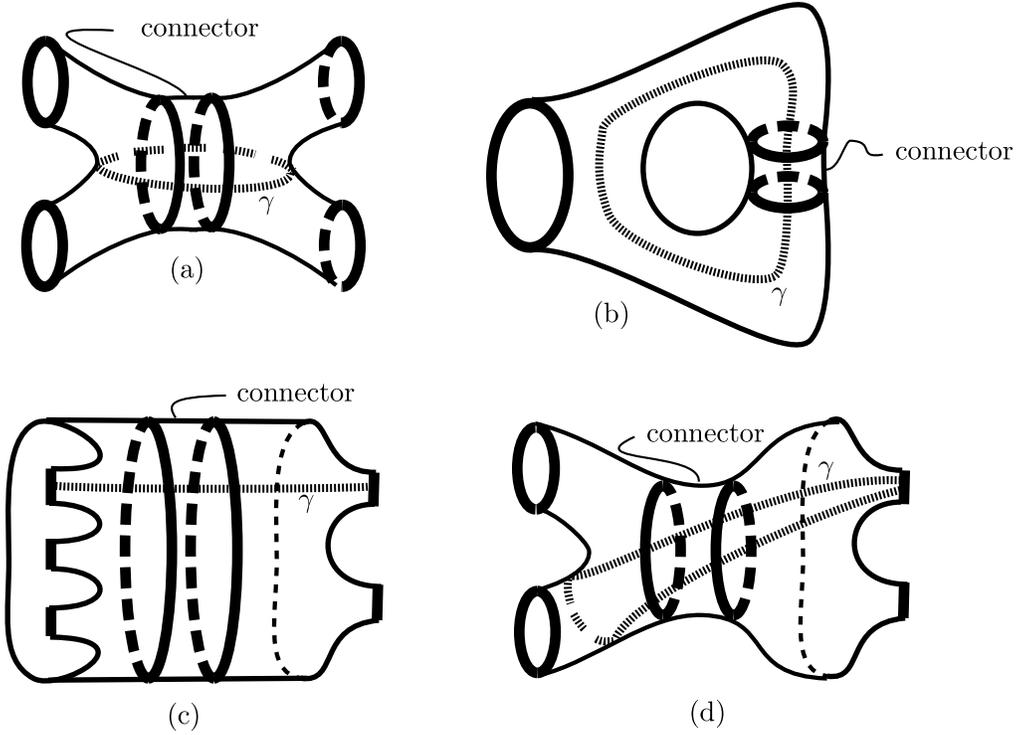}}
\caption{\small The curves $\gamma$ which detect twist parameters.}
\label{PairTwistFig}
\end{figure}

\begin{figure}[H]
\centering
\scalebox{1}{\includegraphics{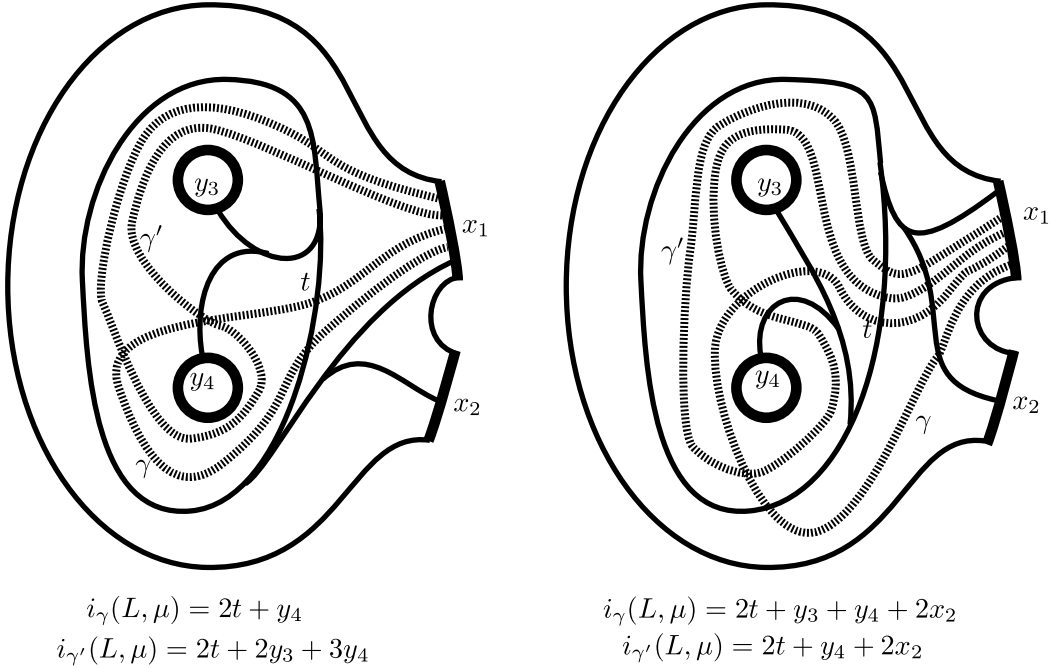}}
\caption{\small Detecting the sense of twist.}
\label{PairTwist2Fig}
\end{figure}

Now for the promised explanation of efficient curves with respect to a train track:  An embedded curve $\gamma$ transverse to a fair train track $\tau$ is {\it efficient} if the train track and curve together do not ``cut off a half-disk or quarter-disk."   More precisely, applying $\pi\inverse$, we may assume $\gamma$ intersects $N(\tau)$ in fibers, and we then define a half-disk in this context as a disk $H$ with an arc of its boundary in $\bdry_hN(\tau)$ and the other arc in $\gamma$.   A quarter disk is a disk $H$ with its boundary decomposed into 3 closed arcs, each sharing an endpoint with another arc, with one arc in $\bdry_hN(\tau)$, a second arc in $\alpha$, and a third arc in $\gamma$.  It is not difficult to show that if $\gamma$ is efficient with respect to $\tau$, then $i_\gamma(\tau(w))=\sum w_i$ where the sum is over weight vector components $w_i$ corresponding to intersections of $\gamma$ with $\tau$, so that if there are $k$ intersections of $\gamma$ with a segment having weight $w_i$, then the weight $w_i$ appears $k$ times in the sum.
\end{proof}

Our next task is to show that the map $\I=(i_\gamma)$ is continuous on $\PM_\D(\bar S,\alpha)$.  The strategy is to express $\PM_\D(\bar S,\alpha)$ as a union of weight cells $\PV(\tau)$ for standard train tracks with respect to $\tau$, and to show that the intersection functions $i_\gamma$, viewed as functions on the weight cells $\V(\tau)$ with $i_\gamma(w):=i_\gamma(\tau(w))$, are continuous on each of these weight cells.  
We will need the following definitions and lemma.

\begin{defn}\label{SplitRespectDef} Suppose $N(\tau')\cup J=N(\tau)$, where $J$ is a product $I$-bundle over a compact surface intersecting $N(\tau')$ in a subset of $\bdry_vN(\tau')\cup \bdry_hN(\tau')$, and $I$ fibers of $J$ are contained in $I$ fibers of $N(\tau)$.  Then $\tau'$ is a {\it splitting} of $\tau$ and $\tau$ is a {\it pinching} of $\tau'$.  

Suppose a lamination $L$ is fully carried by $\tau$ and is embedded transverse to fibers in $N(\tau)$.  After possibly replacing some leaves by the boundaries of their (suitably tapered) regular neighborhoods, we may assume $\bdry_h(N(\tau))\subset L$.   There is an {\it interstitial bundle}, an $I$-bundle which is the completion of the complement of $L$ in $N(\tau)$.   A {\it splitting respecting $L$} of $\tau$ is a splitting $\tau'$ achieved by removing interiors of the fibers of the restriction of the interstitial bundle to a compact submanifold of the base space of the $I$-bundle. ($L$ is then carried by $\tau'$.)

We will use the notation $\tau_1 \split \tau_2$ to indicate that $\tau_1$ is a splitting of $\tau_2$; we write $\tau_2\pinch \tau_1$ to indicate that $\tau_2$ is a pinching of $\tau_1$.  If $L$ is fully carried by $\tau_2$, and  $\tau_1$ is a splitting of $\tau_2$ respecting $L$, then we write $\tau_1\split_L\tau_2$.  
\end{defn}

\begin{lemma}[Splitting Lemma] Suppose $\tau'$ is a splitting of $\tau\embed (\bar S,\alpha)$.  If $\tau$ is fair, then so is $\tau'$.  This means that if $\tau$ has no monogons or 0-gons, then neither does $\tau'$.  Similarly, if $\tau$ is essential so is $\tau'$, which means $\tau'$ also has no (half) Reeb train tracks.
\end{lemma}

The following is adapted from a result in \cite{UO:MeasuredLaminations}.  There is an alternative approach, namely the one in \cite{AH:SurfaceMLS}, for proving the intersection functions $i_\gamma$ are continuous (and piecewise linear) which almost certainly works in our context and might be easier.  
\begin{proposition}\label{ConvexProp} Suppose $(\tau,\bdry \tau)\embed (\bar S,\alpha)$ is a fair train track, and $\gamma$ is the homotopy class of an arc or closed curve as above.  Then for any $\theta\in \H$ the function
$i_\theta:\W_\rationals(\tau)\to \reals$ is convex and has a unique continuous extension to $\intr(\V(\tau))$.
\end{proposition}

\begin{proof}  If $\theta$ is an oriented closed curve in $\delta$, then it is easy to verify that $i_\theta$ is linear, so we assume henceforth that $\theta=\gamma$ is not oriented and not homotopic into $\delta$.

To prove convexity on rational points it is enough to prove convexity on integer points in the cone $\V(\tau)$.  This means it is enough to prove that if $w_0$ and $w_1$ are integer invariant weight vectors, then $i_\gamma(w_0+w_1)\le i_\gamma(w_0)+i_\gamma(w_1)$.    Letting $C_0=\tau(w_0)$ and $C_1=\tau(w_1)$, we can abuse notation, writing $\tau(w_0+w_1)=C_0+C_1=C$.  From this point of view, we wish to prove $i_\gamma(C_0+C_1)\le i_\gamma(C_0)+i_\gamma(C_1)$.   We embed the curve systems $C_0$ and $C_1$ in $N(\tau)$ transverse to fibers and transverse to each other; then one can obtain $C=C_0+C_1$ from $C_0$ and $C_1$ by performing cut-and-paste operations at points of intersection such that one obtains the curve system $C$ transverse to fibers and inducing the weight vector $w_0+w_1$.  If these curve systems are mutually transverse, we obtain a train track $\tau'$ carrying $C_0$ and $C_1$ by pinching the two curves near points of intersection as shown in Figure \ref{PairSplitFig}.  Clearly, then, $\tau'$ is a splitting of $\tau$ and carries $C_0$ and $C_1$.  From the Splitting Lemma we conclude that $\tau'$ is fair.

\begin{figure}[H]
\centering
\scalebox{1}{\includegraphics{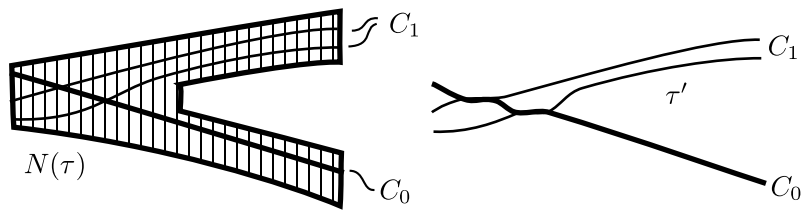}}
\caption{\small Splitting $\tau$ to get $\tau'$ carrying $C_0$ and $C_1$.}
\label{PairSplitFig}
\end{figure}

We can now homotope $\gamma$ to $\gamma_0$ to minimize $i_\gamma(C_0)$ and we can homotope $\gamma$ to $\gamma_1$ to minimize $i_\gamma(C_1)$.   Let $h$ be a homotopy from 
$\gamma_0$ to $\gamma_1$.  Either this is a map $h(t,x)$ where $t\in [0,1]$ and $x\in [0,1]=I$ or $x\in S^1$ depending on whether $\gamma$ is an arc or closed curve.   Thus we can regarding $h$ as a map from an annulus $\bar F$ with $h_0=\gamma_0$ and $h_1=\gamma_1$, or we can regard $h$ as a map $h:(\bar F,\beta)\to (\bar S,\alpha)$, where $\bar F$ is a rectangular disk and $\beta\subset \bdry \bar F$ consists of two closed arcs.  The latter surface pair is a digon.  We can simplify notation by assuming $\bar F=( \bar F,\beta)=(\bar F,\emptyset)$ in case $\bar F$ is an annulus.  Further, we will denote by $\delta_i$ the component of  $ \cl(\bdry \bar F\setminus \beta)$ which is mapped to $\gamma_i$.  We replace $h$ by a map transverse to $\tau'$ and consider the pullback $h\inverse(\tau')$ of $\tau'$, $h\inverse(C_0)$, and $h\inverse(C_1)$ to $\bar F$.      Before proceeding, however, we must modify the map $h$ slightly.  Namely, if we see a 0-gon in the complement of $h\inverse(\tau')$, then $h$ restricted to the boundary of the 0-gon, a closed curve in $\bdry_h(N(h\inverse(\tau'))$ is a map to a component of $\bdry_hN(\tau')$.  If it has degree 0, we can homotope $h$ to eliminate the 0-gon in $h\inverse(\tau')$.  The homotopy either eliminates a simple closed curve from $h\inverse(\tau')$ or it eliminates intersections of $h\inverse(C_0)$ and $h\inverse(C_1)$ on the boundary of the 0-gon.  Otherwise, if $h$ restricted to the boundary of the 0-gon had non-zero degree,  then $\tau'$ would have a 0-gon, a contradiction.  So we can assume $h\inverse(\tau')$ has no complementary 0-gons.  Similarly, if $h\inverse(\tau')$ had a complementary monogon, the boundary would be mapped to a component of $\bdry N(\tau')$ meeting $\bdry_vN(\tau')$ in exactly one fiber, representing a homotopically trivial loop in a complementary component of $N(\tau')$.  This is only possible if $\tau'$ has a complementary monogon, a contradiction.  Therefore we may assume that $h\inverse(\tau')$ has no monogons or 0-gons, i.e. is a fair train track.

 From the pattern $h\inverse(\tau')$, $h\inverse(C_0)$, and $h\inverse(C_1)$ in $\bar F$, we obtain the pullback of $C=\tau(w_0+w_1)$ by cut-and-paste on points of intersection of the two curve systems so that the resulting curve system is carried by $h\inverse(\tau')$.  
We observe that there are no (inessential) arcs of $h\inverse(C_i)$ with both ends in $\delta_i$, otherwise the intersection of $\gamma_i$ with $C_i$ is not minimal.  Also there are no arcs of $h\inverse(C_i)$ with one end in $\beta$ and one end in $\delta_i$, otherwise the intersection of $\gamma_i$ with $C_i$ is not minimal.  If $h\inverse(C_i)$ contains an arc with ends in two different components of $\beta$ or if it contains an essential closed curve in an annular $\bar F$, then $i_\gamma(C_i)=0$, say $i_\gamma(C_0)=0$.  Then $i_\gamma(C_0+C_1)=i_\gamma(C_1)$, which implies what we want: $i_\gamma(C_0+C_1)\le i_\gamma(C_0)+i_\gamma(C_1)$.  So we can assume that neither $h\inverse(C_i)$ contains arcs with ends in $\beta$.  

In order to obtain an upper bound for $i_\gamma(C)$, we will perform cut-and-paste operations at points of intersection of $h\inverse(C_0)$, and $h\inverse(C_1)$ in $\bar F$.   We will do these operations in  a certain order, keeping track of the pattern.   Of course, performing switches on all intersections yields the pullback of $C$.   After performing any finite number of cut-and-paste operations in $\bar F$, we will have a pattern of arcs and closed curves in $\bar F$ consisting of a collection $\E$ of essential arcs in $\bar F$ (with boundary in $\bdry \bar F\setminus\beta$), together with a collection $\A$ of inessential arcs, and a collection $\C$ of closed curves.   We let $\A=\A_0\cup\A_1$ where $\A_i$ consists of inessential arcs with both ends in $\delta_i$.   We observe that initially $|\E|=i_\gamma(C_0)+i_\gamma(C_1)$, where $|\E|$ denotes the number of essential arcs.  The goal is to show that $|\E|$ does not increase when we do cut-and-paste in the appropriate order.   Since $\gamma$ can then be homotoped so that $|C\cap\gamma|=|\E|$, this will finish the proof.  When we perform a cut-and-paste operation at a point of intersection, we also modify the train track $h\inverse(\tau')$ to obtain $\breve\tau$ by splitting the train track on the arc of contact corresponding to the point of intersection.  Thus at every stage of our induction argument, we will have our immersed curve system $\A\cup\E\cup\C$ carried by $\breve\tau$, and $\breve\tau$ will be a fair train track because it is a splitting of the fair $h\inverse(\tau')$.

Inductively,  assuming we have already done some cut-and-paste operations without increasing $|\E|$ and so $\C=\emptyset$, we consider an edgemost arc $a$ of $\A$ among arcs of $\A$ which intersect other arcs or closed curves non-trivially.  If $a$ cuts a half-disk $H$ from $\bar F$, then we choose an edgemost arc $b$ in $H$ of another curve in the current system $\A\cup \E\cup \C$ in $H$.  We perform the cut-and-paste operations at the end(s) of $b$.  Using the fact that $\breve\tau$ is fair, the possibilities are shown in Figure \ref{PairCutPasteFig}, after we rule out possibilities that would imply the existence of 0-gons or monogons.   From these local figures, we conclude that if $b\subset e$, where $e\in \A\cup \E\cup \C$ is another curve of our current immersed curve system, then the cut-and-paste yields arcs isotopic to $a$ and $e$, and $|\E|$ does not change.  Furthermore, we do not introduce closed curves, so $\C=\emptyset$ remains true.  We repeat the type of cut-and-paste described above until no arc of $\A$ intersects any other curve.

\begin{figure}[H]
\centering
\scalebox{1}{\includegraphics{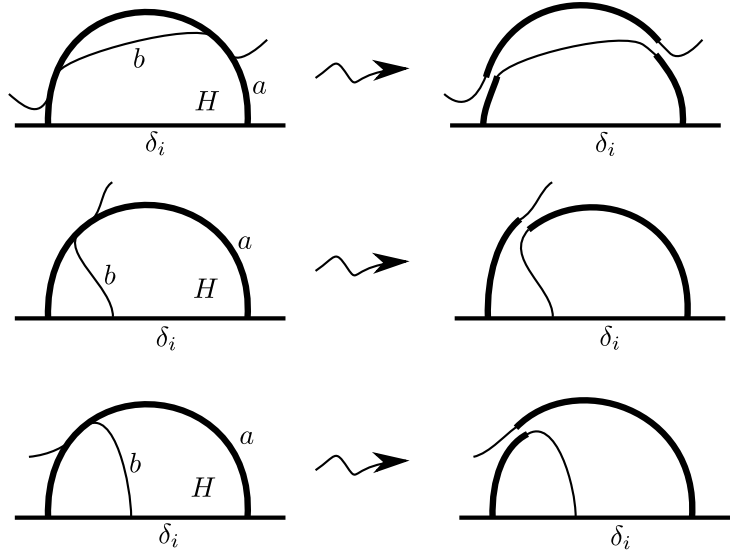}}
\caption{\small Possibilities for a cut-and-paste operation involving an edgemost arc $a$.}
\label{PairCutPasteFig}
\end{figure}

Now ignoring the arcs of $\A$, performing a cut-and-paste at a point of intersection of two other arcs, which which must both be in $\E$, we do not increase $|\E|$, since this was just half the number of points of $\E\cap\bdry \bar F$, and the boundary points are unchanged by the operations.  Clearly, no new closed curve is introduced, so $\C$ remains empty.   Of course, the operation may introduce inessential arcs, so typically $|\E|$ becomes strictly smaller and $|\A|$ can increase.  If we have introduced a new arc of $\A$ which intersects other curves, we return to the algorithm for eliminating these intersections.  We repeat these operations until no intersections remain, and we obtain the $h\inverse(C)$ in $\bar F$.   Now it is clear that we can homotope $\gamma$ so that the intersection with $C$ is $|\E|$ and we have finished the proof that $i_\gamma(C_0+C_1)\le i_\gamma(C_0)+i_\gamma(C_1)$.  As we observed, this proves the required convexity on $\W_\rationals(\tau)$.

Proving that there is a unique continuous extension of $i_\gamma$ to $\intr(\V(\tau)$, given the convexity on rational points, is elementary analysis.
\end{proof}

The next task is to show that the unique continuous extension described above actually gives $i_\gamma(w)$ at non-rational points.

\begin{proposition}  Suppose $(\tau,\bdry \tau)\embed (\bar S,\alpha)$ is a fair train track, and $\gamma$ is the homotopy class of an arc or closed curve as above.  Then the function
$i_\gamma:\V(\tau)\to \reals$ is convex and therefore continuous on $\intr(\V(\tau))$.
\end{proposition}

\begin{proof}  To prove this, suppose we have a measured lamination $\tau(w)=(L,\mu)$ fully carried by $\tau$.  We assume that $L$ is embedded transverse to fibers in $N(\tau)$, and after possibly replacing some leaves by the boundaries of their (suitably tapered) regular neighborhoods, we may assume $\bdry_h(N(\tau))\subset L$.   We know that we can realize $i_\gamma(L)$ as an actual intersection by homotoping $\gamma$.   Now $N(\tau)$ can be subdivided into rectangular products of the form $P=s\times I$, where $s$ is a segment of the train track.  For every such product, we homotope $\gamma$ further (without increasing the intersection with $L$) such that every component of $\gamma\cap P$ becomes vertical, i.e. a subset of a fiber $\{x\}\times I$, or ``horizontal", i.e. contained in the interstitial bundle.  We extend the homotopy to modify the intersections with complementary components of $L$ as necessary. We split $N(\tau)$ by removing fibers of the interstitial bundle intersected by $\gamma$ (and split a bit further) obtaining a splitting $\tau'$ of $\tau$ which fully carries $L=\tau'(w')$.  Now $\gamma$ intersects $\tau'$ vertically, which means that one can apply the projection $\pi:N(\tau')\to \tau'$ to replace $\gamma$ by a curve which is transverse to $\tau'$.   In fact, it is easy to check that there do not exist any half-disks $(H,\beta)$ with $\beta\subset \bdry_hN(\tau')$ and with the complementary boundary arc in $\gamma$.   In other words, $\gamma$ is efficient with respect to $\tau$, see the proof of Proposition \ref{InjectiveProp}. 
This can be used to show that $i_\gamma(\tau'(v))=\sum_iv_i$ where the sum is over intersections of $\gamma$ with $\tau'$, and adds the weights at each intersection.  Hence $i_\gamma$ is linear on $\V(\tau')$, and continuous.  The fact that $\tau'$ is a splitting of $\tau$ implies that there is a linear map $\V(\tau')\to \V(\tau)$ which maps $w'$ to $w$ so that $\tau'(w')=\tau(w)$ as a measured lamination.  This shows that $i_\gamma$ is linear on a convex subset of $\V(\tau)$, which is the convex hull of rational points.  Thus, $i_\gamma$ must be the same as the continuous extension described in the previous proposition, on this convex hull.  Since we started with arbitrary $w\in \intr(\V(\tau))$, we have proved the proposition.
\end{proof}

One shortcoming of Proposition \ref{ConvexProp} is that we have not proved that $i_\gamma$ is continuous on all $\W_\rationals(\tau)$.  For essential laminations in 3-manifolds carried by branched surfaces, there is a similar convexity result, see \cite{UO:MeasuredLaminations}, and there are examples to show that $i_\gamma$ can be discontinuous at $\bdry \V(\tau)$.   In our setting, we can strengthen the result to obtain continuity on all of $\V(\tau)$.   The first step towards this strengthening is the following lemma.

We will say a weighted curve system $K$ is {\it maximal} if it is not possible to embed another curve disjointly from those in $K$, unless the new curve is isotopic to one of the curves in $K$.

\begin{lemma} \label{MaximalLemma} Suppose $K$ is a maximal curve system fully carried by a fair (or essential) train track $\tau$ in a surface pair $(\bar S,\alpha)$.  Then the complementary surfaces with cusps for $\tau$ are of the following types $(\bar F,\beta)$

(i) $\bar F$ a disk with $\beta$ consisting of 2 or  3 arcs on its boundary, i.e. $(\bar F,\beta)$ a digon or trigon,

(ii) $\bar F$ an annulus with $\beta$ consisting of at most one arc on its boundary,

(iii) $\bar F$ a pair of pants with $\beta=\emptyset$.

\noindent If $\tau$ is a good train track, we can of course rule out annuli and digons as well.
\end{lemma}

\begin{proof}  This is an easy exercise.  If there is a complementary surface with cusps which is not on this list, we can regard it as a surface pair $(\bar F,\beta)$, where $\beta$ may include arcs or closed curves in $\alpha$.  There is an essential curve in $(\bar F,\beta)$, which can easily be extended to an essential curve in $(\bar S,\alpha)$ by paralleling an existing curve of $K$.
\end{proof}

\begin{lemma}\label{ContinuationLemma}  Suppose $\tau$ is an essential train track in $(\bar S,\alpha)$, and suppose $(L,\mu)=\tau(w)$ is fully carried by a sub train track $\hat \tau$ of $\tau$.  Suppose $\tau$ fully carries a maximal curve system.  Suppose $(K,\nu)=\tau(u)$ is any rational measured lamination fully carried by $\tau$, representing a maximal curve system and sufficiently close to $w$ in $\V(\tau)$ (or $\PV(\tau)$).   Then there exists $\tau_1\split_K\tau$,  $\tau_2\pinch \tau_1$ such that:

(i) $\hat \tau$ remains a subtrain track of $\tau_1$, and carries $K$, but the pinching of $\tau_1$ to $\tau_2$ also pinches $\hat\tau$ to yield $\hat \tau_2$,

(ii) we can express $\tau_2$ as a disjoint union $\tau_2=\hat\tau_2\sqcup \breve\tau_2$ where $\hat\tau_2$ is the pinching of $\hat\tau$ and fully carries $L$  while $\breve\tau_2$ is a collection of closed curves.  

(iii) $K$ can be decomposed as $\hat K\sqcup \breve K$ where $\hat K$ is fully carried by $\hat\tau_2$ and $\breve K$ is fully carried by $\breve\tau_2$.  (The only difference between $\breve\tau_2$ and $\breve K$ is that $\breve K$ is a weighted curve system while $\breve\tau$ is a curve system.)
\end{lemma}

\begin{proof}  We embed the measured lamination $K=\tau(u)$ in $N(\tau)$ transverse to fibers with $\bdry_hN(\tau)\subset K$ after possibly replacing $K$ by the boundary of its regular neighborhood.  The proof will show how close $u$ must be to $w$ in $\V(\tau)$.  Note that the interstitial bundle $J$ for $K$ is compact if we use atomic measures on curves of $K$.  If $\pi:N(\tau)\to \tau$ is the projection, then we eliminate a portion of the interstitial bundle $J$ from $N(\tau)$ to achieve a splitting yielding $N(\tau_1)$, namely we remove $J-\pi\inverse(\hat\tau)$.    We then observe that $\tau_1-\hat\tau$ is a curve system, not necessarily properly embedded in $(\bar S,\alpha)$.  Some arcs of $\tau_1-\hat\tau$ may be attached to $\hat \tau$.   For convenience we enlarge $\hat \tau$, replacing it by $N(\hat \tau)$, so that we can think of the arcs of $\tau_1-\hat\tau$ as lying in the closure of the complement of $N(\hat \tau)$.  This first splitting has no effect on $\hat \tau$.

We consider possible types of arcs $b$ in $\tau_1-N(\hat\tau)$.  They can be classified by where each the ends lie:  (A) in an interval component of $\bdry_hN(\hat \tau)$;  (B) in a closed curve component of $\bdry_hN(\hat\tau)$; (C) in an interval component of $\alpha-N(\hat\tau)$; (D) in a closed curve component of  $\alpha-N(\tau)$.    There are different types of interval components of $\bdry_hN(\hat\tau_1)$, namely:  (A1) an arc with both ends in $\bdry_vN(\hat \tau)$; (A2) an arc with both ends in $\alpha$; (A3) an arc with one end in $\alpha$ and one end in $\bdry_vN(\hat\tau)$.   There are then 49 different types of arcs $b$ to consider.   In fact there are even more, depending on the sense of branching where $b$ is attached to $N(\hat \tau)$.  Our goal is to modify $\tau_1$ by pinching and $K$-splitting to render all arcs $b$ S-shaped as shown in Figure \ref{PairSFig}(a) (or the mirror image) while eliminating arcs $b$ which are U-shaped as shown in  Figure \ref{PairSFig}(b) (or the mirror image).

\begin{figure}[H]
\centering
\scalebox{1}{\includegraphics{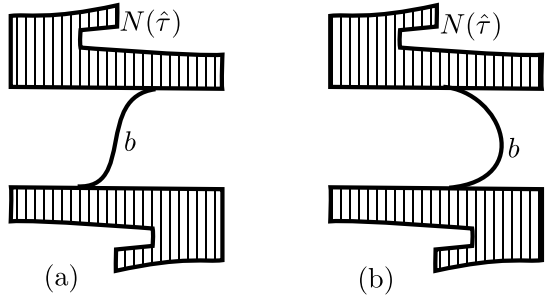}}
\caption{\small Arcs of $\tau_1\setminus \hat \tau$.}
\label{PairSFig}
\end{figure}

As a first step, if we have a U-shaped arc joining two $S^1$ components of $\bdry_hN(\hat \tau)$ then $ \tau_1$ contains a Reeb train, hence also $\tau$ contains a Reeb train track, a contradiction.   Here we are using Lemma \ref{MaximalLemma} to ensure that the U-shaped arc together with the two $S^1$'s are embedded in an annular subsurface of $S$.  We are also using the Splitting Lemma.

Next we describe a splitting which eliminates a U-shaped arcs if at least one end is directed toward either a component of $\bdry_vN(\hat \tau)$ or towards a component of $\alpha$.  If the initial weights of $u$ on segments of $\hat \tau$ are very large compared to the weights on segments of $\tau\setminus\hat\tau$, then clearly the same is true for the corresponding weights on $\hat\tau$ and $\tau_1$. Then the $K$-splitting shown in Figure \ref{PairRemoveUFig}(a), (b), or (c) eliminates the $U$-arc.  Figure \ref{PairRemoveUFig}(c) indicates that the  the  $K$-splitting  may remove a product of the interstitial bundle with both ends in $\bdry_vN(\tau_1)$, and may thus join two arcs of $\tau_1\setminus\tau$ or change an arc to a closed curve.     We remove all $U$-arcs of this kind using $K$-splitting to obtain a new $\tau_1$.  These splittings do not affect $\hat\tau$.  

\begin{figure}[H]
\centering
\scalebox{1}{\includegraphics{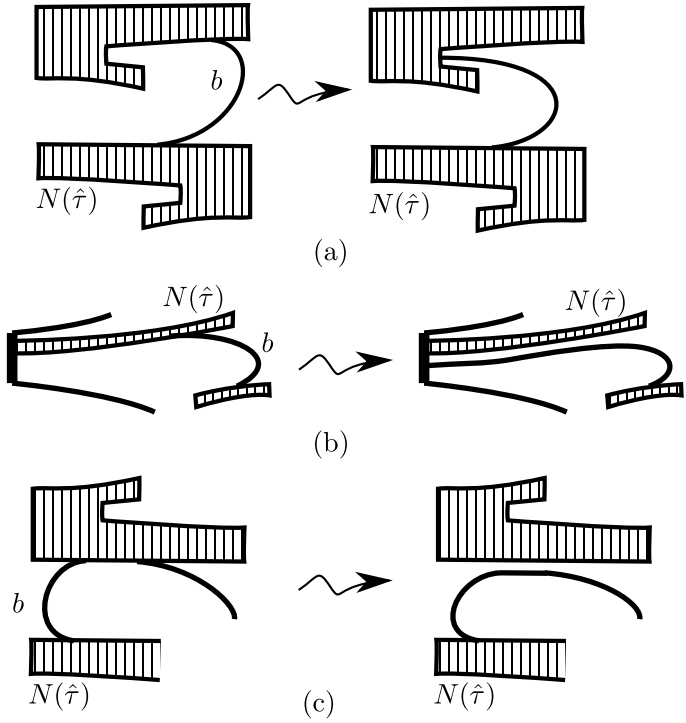}}
\caption{\small Splitting to remove U-shaped arcs $b$.}
\label{PairRemoveUFig}
\end{figure}

As a result of the splitting operations, arcs of $\tau_1-N(\hat\tau)$  will have at least one end in $\beta$ of the complementary surface pair $(\bar F,\beta)$ for the train track $\hat\tau$.  
We an appropriate pinching near each end of such an arc  to ensure that only S-arcs are introduced.  This can be be done 
by consistently pinching in the same sense around components of $\bdry F$, with the sense coming from the induced orientation on boundaries of $F$, see Figure \ref{PairSenseFig}.

\begin{figure}[H]
\centering
\scalebox{1}{\includegraphics{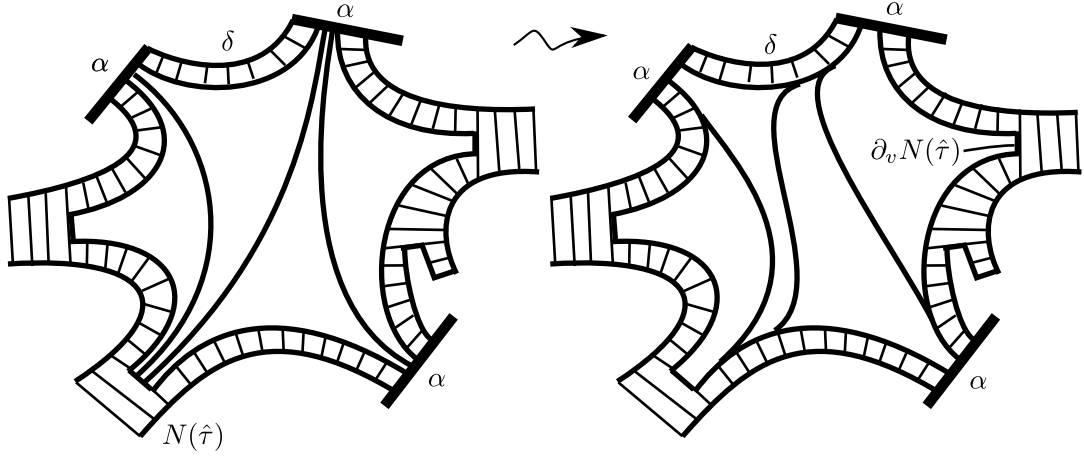}}
\caption{\small Repinching to get S-shaped arcs.}
\label{PairSenseFig}
\end{figure}

\begin{figure}[H]
\centering
\scalebox{1}{\includegraphics{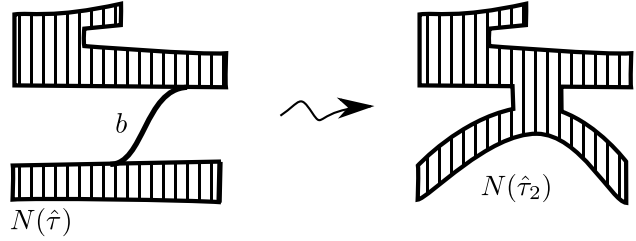}}
\caption{\small Pinching to eliminate S-shaped arcs.}
\label{PairSPinchFig}
\end{figure}

Finally, we pinch near $S$-arcs as indicated in Figure \ref{PairSPinchFig} to eliminate the $S$ arcs.  These pinchings together with the previous ones, yield $\tau_2$ as a pinching of $\tau_1$, and $\hat\tau_2$ the induced pinching of $\hat \tau$.  The remainder $\breve\tau_2=\tau_2\setminus \hat\tau_2$ consists of closed curves only.    Thus $K$ can be decomposed as a union $\hat K$ of closed curves fully carried by $\hat\tau_2$ together  with $\breve K$ a collection of closed curves carried by $\breve \tau_2$.  
\end{proof}

\begin{proposition}  Suppose $(\tau,\bdry \tau)\embed (\bar S,\alpha)$ is a fair train track, and $\gamma$ is the homotopy class of an arc or closed curve as above.  Then the function
$i_\gamma:\V(\tau)\to \reals$ is convex and  continuous on $\V(\tau)$.
\end{proposition}

\begin{proof}  We use Lemma \ref{ContinuationLemma}. Suppose $w\in \bdry \V(\tau)$, $w\ne 0$.  Then $L=\tau(w)$ is carried by a sub train track $\hat \tau$.  By splitting and pinching as in the lemma we obtain $L$ carried by $\tau_2=\hat\tau_2\cup \breve\tau_2$, fully carried by $\hat \tau_2$.   The measured lamination $L$ can be written as $\hat\tau_2(\hat w)$.   On the other hand the lamination $K$ of the lemma can be written as $\hat\tau_2(\hat u)\cup \breve\tau_2(\breve u)$.  Clearly, since $\hat\tau_2$ and $\breve\tau_2$ are disconnected, an intersection function $i_\gamma$ satisfies $i_\gamma(\hat\tau_2(\hat u)\cup \breve\tau_2(\breve u))=i_\gamma(\hat\tau_2(\hat u))+i_\gamma( \breve\tau_2(\breve u))$. Since the splitting and pinching give a  linear map from a subspace of $\V(\tau)$ to $\V(\hat\tau_2)$, we obtain the desired continuity of $i_\gamma$ at the arbitrary boundary point $w$ of $\V(\tau)$.  In detail, letting $L=\hat\tau_2(\hat w)$, we get $i_\gamma[\tau((1-t)u+tw)]=i_\gamma[\breve\tau_2((1-t)\breve u]+i_\gamma[\hat\tau_2((1-t)\hat u+t\hat w)]$, and we prove continuity by letting $t$ approach 1.  The intersection $i_\gamma[\hat\tau_2((1-t)\hat u+t\hat w)]$ then converges to $i_\gamma[(\hat\tau_2(\hat w)]$ from the convexity, hence continuity of $i_\gamma$ in the interior of $\V(\hat\tau_2)$. \end{proof}

\begin{proof}[Proof of Theorem \ref{MainThm}]  (We will deal with the projective space $\PM(\bar S,\alpha)$  here.)  The standard train tracks with respect to the decomposition $\D$ form a closed system of train tracks.  This means that one can form a quotient of all the projective weight cells $\PV(\tau)$, $\tau\in \T$ in a natural way by identifying a face (of any dimension) $\PV(\hat\tau)$ of $\PV(\tau_1)$ corresponding to a sub train track $\hat \tau$ of $\tau_1$ with a face  $\PV(\hat\tau)$ of $\PV(\tau_2)$ corresponding to a sub train track $\hat \tau$ of $\tau_2$ whenever $\tau_1$ and $\tau_2$ have a common sub train track $\hat \tau$.  This gives the quotient which we called $\PM_\D(\bar S,\alpha)$.  Thus pasting the continuous restrictions of $\I=(i_\gamma)$ to the weight cells of standard train tracks, we obtain a continuous bijection from $\PM_\D(\bar S,\alpha)\to \PM(\bar S,\alpha)$.  This proves the two spaces are homeomorphic.
\end{proof}

\nocite{JLH:Stability}
\bibliographystyle{amsplain}
\bibliography{ReferencesUO3}
\end{document}